\documentclass[10pt,twoside]{article}
\usepackage[all]{xy}
\usepackage{amsthm,amsmath,amssymb,amsfonts,graphicx,color,enumitem,fancyhdr}
\usepackage{hyperref}
\hypersetup{
    colorlinks=true,
    linkcolor=blue,
    filecolor=magenta,
    urlcolor=cyan,}
  \newtheorem{thm}{Theorem}[section]
  \newtheorem{lem}[thm]{Lemma}
  \newtheorem{prop}[thm]{Proposition}
  \newtheorem{cor}[thm]{Corollary}
  \theoremstyle{definition}
  \newtheorem{defn}[thm]{Definition}
  \newtheorem{exm}[thm]{Example}
  \newtheorem{rmk}[thm]{Remark}
  
 \newcommand\ra{\rightarrow}

\newcommand{\id}{\mbox{\rm Id}}
\newcommand{\reg}{\mbox{\rm reg}}
\newcommand{\ord}{\mbox{\rm ord}}
 \numberwithin{equation}{section}
\title{\textbf{Application of H{\"o}hle's Square Roots on Hoop  Algebras}}
\author{Ali Madanshekaf$^{^{1}} \footnote{Corresponding Author: Ali Madanshekaf}$  \and  Mohammad Mahdi Motamedi Nezhad$^{^{1}}$ \\
	{\small\em $^1$ Department of Mathematics,} \\ {\small\em Faculty of Mathematics, Statistics and Computer science,} \\ {\small\em Semnan University, Semnan,  Iran}  \\
	{\small\tt \href{amadanshekaf@semnan.ac.ir}{amadanshekaf@semnan.ac.ir}  and  \href{motamedinezhad@semnan.ac.ir}{motamedinezhad@semnan.ac.ir} } }
\date{}

\begin{document}

\maketitle

\begin{abstract}
Square root is a useful tool to study the properties of (ordered) algebraic   structures. In this article, we are going to employ this tool to study hoop algebras. To do so, we define square root and make the first attempt to explore the significance properties of this concept in this setting. Then, due to the key role of square roots in obtaining new hoop algebras, we apply  them on the filters of hoop algebras, and show that the formation of square roots on quotient structures of hoop algebras by their filters is well-behaved. In addition, a new class of hoop algebras having square roots, so called good hoop algebras, is introduced  and its relationships with other classes of ordered algebras such as Bollean algebras and Godel algebras are explored. Several examples are provided as well. Ultimately,  it is shown that the class of all bounded hoop algebras with square roots is a variety.

\end{abstract}




\section{Introduction}
Hoop algebras are naturally ordered commutative residuated integral monoids, introduced by B. Bosbach \cite{Bos1,Bos2}. He showed that the resulting class of structures can be viewed as an equational class, and that the class is congruence distributive and congruence permutable. In a manuscript by J. R. B\"{u}chi and T. M. Owens (\cite{BuOw}) devoted to a study of Bosbach's algebras, written in the mid-seventies, the commutative members of this equational class were given the name \textit{hoops}.
Over the past 50 years, many scientists and researchers have studied the properties of these algebras and numerous articles have been published. 

Among the researchers who have done research in this field so far, the following authors can be mentioned.
Blok and Ferririm \cite{BlFe1,BlFe2}, regarding  on varieties and quasi-varieties of hoops and their reducts and the structure of hoop algebras, Cignoli and Torrens \cite{CiTo}, regarding the free cancellative hoops, Berman and Blok~\cite{BeBl}, regarding the free {\L}ukasiewicz and hoop residuation algebras, Agliano, Ferririm and Montagna \cite{AgFeMo}, regarding the basic hoops, Esteva and {et al. \cite{EsGoHaMo}, regarding the hoops and fuzzy logic, Georgescu and et al. \cite{GeLePr}, regarding the pseudo-hoops, Kondo \cite{Kon}, regarding some types of filters in hoops. In recent years, multiple articles have been published on various topics related to hoop algebras, such as different types of filters, various types of ideals,  derivatives, (semi) topological hoops, the relationship between hoop algebras and other logical algebras, and so on (see also \cite{ABAK1}). 
Square roots are valuable tools for algebraic structures with non-idempotent binary operations. They can provide us useful information about the elements and some properties of these types of algebras, and they help us to build new operations from binary operations according to our needs. For this reason, square roots have many applications in the theory of semigroups, groups, and rings.

In \cite{Hol}, H{\"o}hle also attempted to explore the significance of square roots in the class of residuated lattices; integral, commutative, residuated, $\ell$-monoids. He classified the class of $MV$-algebras with square root and presented a decomposition for each $MV$-algebra as a direct product of a Boolean algebra and a strict $MV$-algebra. He also found the connection between the concepts of divisibility and injectivity and strictness on complete $MV$-algebras. His study showed that a complete $MV$-algebra with square roots is either a complete Boolean algebra or a Boolean valued model of the real unit interval viewed as an $MV$-algebra or a product of both. Ambrosio~\cite{Amb} continued to study square roots on MV-algebras. She introduced the concept of 2-atomless MV-algebras and showed that strict $MV$-algebras are equivalent to 2-atomless $MV$-algebras. Then she proved that each strict $MV$-algebra contains a copy of the $MV$-algebra of all rational dyadic numbers. Finally, she investigated the homomorphic image of $MV$-algebra with square roots. 
 B\v{e}lohl\'{a}vek studied the class of all residuated lattices with square roots and showed that it is an equational class \cite{Rbe}.

The main goals of this paper are to provide a study on hoop algebras with square roots and present their main properties:
\begin{itemize}
	\item Define square root on hoop algebras and review some  examples, 
	\item Studying properties of square roots on variant hoop algebras, such as bounded and basic hoops, and  hoops with double negation property (DNP, for short), and etc.
	\item Define  $n$-th roots as extending  square root and examine some of its properties,
	\item The application of  square root on the filters of hoop algebras is investigated and under some circumstances, a structural  theorem regarding the effect of the square root on quotient hoop algebras has been proved (see Theorem~\ref{5.14} below),
	\item Define good hoop algebra  and explore some of its characteristics and then its relationships with   Boolean algebras and G\"{o}del algebras have been investigated (see Theorems~\ref{4.4} and \ref{Godel algebra}),
	\item In the end it is shown that the class of all bounded hoop algebras with square roots is a variety.
\end{itemize}
The paper is organized as follows. Section 2 contains basic definitions, properties, and results about hoop algebras that will be used in the next sections. Section 3 while define square root on hoop algebras, we study properties of square roots on very hoop algebras. In Section 4, firstly  applications of  square root on the filters of hoop algebras are provided and then we define the concept of a good hoop algebra and strict square root and classify the class of some hoop algebras with square roots.
\section{Preliminaries}
In this section, while providing definitions and some properties of hoop  algebras  we collect the prerequisites used for the rest of the paper.
\begin{defn}\label{1.1}{\rm \cite{AgFeMo,Kon}}
A hoop  algebra is an algebra $\mathbf{H}=(H;\odot,\rightarrow,1)$ of type $(2,2,0)$ such that 

{\rm (H1)} $(H;\odot,1)$ is a commutative monoid,

{\rm (H2)} $x\rightarrow x=1$,

{\rm (H3)} $x\odot (x\rightarrow y)=y \odot (y \rightarrow x)$,

{\rm (H4)} $(x \odot y)\rightarrow z = x\rightarrow (y\rightarrow z)$, \\
for every $x,y,z\in H$.
\end{defn}
Let $\mathbf{H} = (H;\odot_H,\rightarrow_H,1_H) $ and $\mathbf{G} = (G;\odot_G,\rightarrow_G,1_G) $ be two hoop  algebras. We say that $G$ is a subalgebra of $H$, denoted by $\mathbf{G} \le \mathbf{H}$, if $G\subseteq H$ and every fundamental operation in $G$ is a restriction of the corresponding operation in $H$, meaning that $1 \in G$,  $\odot_G = \odot_H\mid_G$, and $\rightarrow_G = \rightarrow_H\mid_G$.

In a hoop  algebra $\mathbf{H} = (H;\odot_H,\rightarrow_H,1_H) $, we define $x\le y$ if and only if  $x\rightarrow y=1$. It is easily seen that  $\le$ defines a partial order relation on $H$ (see also~\cite{AgFeMo}). We define $x^0=1$ and $x^n=x^{n-1}\odot x$ for every $n\in \mathbb{N}$. 	An element $x\in H$ is called idempotent if and only if $x\odot x= x$. The set of idempotent elements in $H$ is denoted by  $\id(H)$. $H$ is called idempotent if each of its elements is idempotent. $H$ is called bounded if it has a least element denoted by $0$. We define the unary  operation ``$\;'\; $"  for every $x\in H$ as $x': =x\rightarrow 0$. We also set $(x')'= x''$.
 An element $x\in H$ is called dense if and only if $x'= 0$. The set of all dense elements in $H$ is denoted by $D(H)$.
 An element $x\in H$ is called regular if $x''= x$. The set of regular elements in $H$ is denoted by $\reg(H)$. If for every $x \in H$, $x''= x$, then we say that the bounded hoop  algebra $H$ has the ``double negation property", or simply $(DNP)$. Let $H$ be a bounded hoop  algebra. For every $x\in H$ with $x\ne 1$, the order of $x$, denoted by $\ord (x)$, is the smallest $n\in \mathbb{N}$ such that $x^n = 0$, if such an $n$ exists. In this case, $x$ is called a nilpotent element. The set of all nilpotent elements in $H$ is denoted by $N(H)$. If $x^n\ne 0$ for every $n$, then we define $ \ord (x) = \infty $. The order of $H$, denoted by $\ord (H)$, is defined as $\ord(H) = \sup\{ \ord(x) :x\in H-\{1\}\}$.     

We call a bounded hoop  algebra local if $ \ord(x) < \infty$ or $ \ord (x') < \infty$ for every $x\in H$. A bounded hoop  algebra is called locally finite if the order of every $x\in H$ and $x\ne 1$, is finite.
\begin{exm}\label{1.51}
a) Let $\mathbf{G}=(G;+,-,0,\wedge,\vee)$ be an Abelian $l$-group. Consider the positive cone $P(\mathbf{G})$ of $\mathbf{G}$, defined by $P(\mathbf{G}) = \{x\in G: x\vee 0=x\}.
$ 
It is clear that $\mathbf{P(G)}: = (P(\mathbf{G});\cdot,\rightarrow,1)$ is an algebra with the following operations:
\[
x\cdot y:=x+y, \quad x\rightarrow y:=(y-x)\vee 0, \quad 1_{\mathbf{P(G)}}:=0_\mathbf{G}
\]
It can be easily checked that $\mathbf{P(G)}$ satisfies the conditions of Definition \ref{1.1}. Therefore, it is a hoop  algebra. Notice that the partial order on $\mathbf{P(G)}$ is the converse of the  partial order induced by $\mathbf{G}$.\\
b) Consider the free monoid generated by a single element $a$. Put $C_a= \{ 1=a^0, a, a^2, \cdots \}$ 
and define the operations $\odot$ and $\rightarrow$ on $C_a$ for any $n,m <w$ as follows:
\[a^n\cdot a^m=a^{n+m}, \quad a^n\rightarrow a^m=a^{\max(m-n,0)}\] 
It is evident that $C_a$ satisfies the conditions of Definition \ref{1.1}. Therefore, it is a hoop algebra.
For more information see~\cite{BlFe1}. \\
c) Let $\mathbf{G}= (G;+,-,0,\wedge,\vee) $ be an Abelian $l$-group, $u\in G$ be an arbitrary element with $u \ge 0$.
 Put $[0,u]=\{x\in G : 0\le x\le u \}.$ 
On the set $[0,u]$, define the operations $\odot$ and $\rightarrow$ for every $x,y\in [0,u]$ as follows:
\[
x \odot y = (x+y-u)\vee 0, \quad x\rightarrow y = (y-x+u)\wedge u
\]
It is easy to see that $\Gamma(G,u) := ([0,u];\odot,\rightarrow, u) $ satisfies the conditions of Definition \ref{1.1}. Therefore, it is a bounded hoop algebra (see also \cite{GeLePr}).\\
d)	
Let $H=\{0,a,b,c,d,1\}$. Define the operations $\odot$ and $\rightarrow$ on $H$ as follows:
\[
{\hfill
\hbox{
		\label{tabel1}
		\begin{tabular}{|c||c|c|c|c|c|c|} 
			\hline
			$\odot$ & 0 & a & b & c & d & 1\\ 
			\hline
			0 & 0 & 0 & 0 & 0 & 0 & 0 \\ 
			\hline
			a & 0 & b & b & d & 0 & a \\ 
			\hline		
			b & 0 & b & b & 0 & 0 & b \\ 
			\hline		
			c & 0 & d & 0 & c & d & c \\ 
			\hline		
			d & 0 & 0 & 0 & d & 0 & d \\ 
			\hline
			1 & a & b & c & d & e & f \\
			\hline	
		\end{tabular}
}
}\qquad
{\hfill \hbox{
		\begin{tabular}{|c||c|c|c|c|c|c|} 
			\hline
			$\rightarrow$ & 0 & a & b & c & d & 1 \\ 
			\hline
			0 & 1 & 1 & 1 & 1 & 1 & 1 \\ 
			\hline
			a & d & 1 & a & c & c & 1 \\ 
			\hline		
			b & c & 1 & 1 & c & c & 1 \\ 
			\hline		
			c & b & a & b & 1 & a & 1 \\ 
			\hline		
			d & a & 1 & a & 1 & 1 & 1 \\ 
			\hline
			1 & 0 & a & b & c & d & 1 \\
			\hline	
		\end{tabular}
} 
\hfill
}
\]
It can be easily checked that the introduced structure satisfies the conditions of Definition \ref{1.1} Therefore, the algebra is a bounded hoop algebra.
\end{exm}
	Let $H$ and $G$ be two bounded hoop  algebras. A map $f: H \rightarrow G$ is called a hoop  homomorphism if  for every $x,y\in H$, $f(0) = 0$, $f(1) = 1$, and $f(x\odot y) = f(x)\odot f(y)$, $f(x\rightarrow y) = f(x)\rightarrow f(y)$. Then one can prove that it is order preserving. In addition, if $f$ is one to one then, it is an order embedding. Indeed, for $x, y\in H$ we have $x\le y$ iff $x\rightarrow y=1$ iff $f(x\rightarrow y)=f(1)=1$ iff $f(x) \rightarrow f(y)=1$ iff $f(x)\le f(y)$.

\begin{prop}\label{1.15}
	If $(H;\odot_H,\rightarrow_H,1_H)$ and $(G;\odot_G,\rightarrow_G,1_G)$ are two hoop  algebras, then $(H\times G;\odot_{H\times G},\rightarrow_{H\times G},(1_H,1_G))$ becomes a hoop  algebra with the following operations:
\begin{align*}
		(x_1,x_2) \odot_{H\times G} (y_1,y_2)&= (x_1\odot_H y_1, x_2\odot_G y_2),\quad \\
		(x_1,x_2) \rightarrow_{H\times G} (y_1,y_2)&= (x_1 \rightarrow_H y_1, x_2\rightarrow_G y_2) 
\end{align*}
for every $x_1,y_1 \in H$ and $x_2,y_2\in G$.
\end{prop}

\begin{proof}
	It is straightforward to verify that the conditions of Definition \ref{1.1} hold.
\end{proof}
\begin{thm}\label{1.4}{\rm (\cite{Bos1,Bos2})}
Let $(H;\odot,\rightarrow,1)$ be a hoop  algebra. Then for every $x,y,z,a\in H$ we have:
\begin{itemize}
	\item[{\rm (i)}] $(H;\le)$ is a $\wedge$-semilattice with  $x\wedge y=x\odot(x \rightarrow y)$,

	\item[{\rm (ii)}] $x\odot y \le z$ if and only if $x\le y\rightarrow z$,

	\item[{\rm (iii)}] $ x\odot y \le x, y$,

	\item[{\rm (iv)}] $ x \odot (x \rightarrow y) \le x,y $,

	\item[{\rm (v)}] $ x \le y \rightarrow x $,

	\item[{\rm (vi)}] $ 1 \rightarrow x = x $,

	\item[{\rm (vii)}] $ x \rightarrow 1 = 1 $, 

	\item[{\rm (viii)}] $ y \le (y\rightarrow x)\rightarrow x $,

	\item[{\rm (ix)}] $ x \le (x\rightarrow y) \rightarrow x $, 

	\item[{\rm (x)}] $ x \le y \rightarrow (x \odot y) $,

	\item[{\rm (xi)}] $ x \rightarrow y \le (y \rightarrow z) \rightarrow (x \rightarrow z) $,

	\item[{\rm (xii)}] $ (x \rightarrow y) \odot (y \rightarrow z) \le x \rightarrow z $,

	\item[{\rm (xiii)}] $x \le y $ implies $ x\odot z \le y \odot z $,

	\item[{\rm (xiv)}] $ x \le y $  implies $ z \rightarrow x \le z \rightarrow y $, 

	\item[{\rm (xv)}] $x \le y $  implies  $ y \rightarrow z \le x \rightarrow z$. 
\end{itemize}
\end{thm}

\begin{prop}\label{1.6} {\rm \cite{ABAK1,GeLePr}}
Let $H$ be a hoop  algebra. For every $x,y\in H$, we define
\[
x \vee y = ((x\rightarrow y )\rightarrow y) \wedge ((y\rightarrow x) \rightarrow x)
\]
Then, for every $x,y\in H$, the following conditions are equivalent:

{\rm (i)} The operation $\vee$ is associative.

{\rm (ii)} $x\vee (y \wedge z)\le (x \vee y)\wedge (x\vee z)$ for all $x,y,z \in H$.

{\rm (iii)} $x\le y$ implies $x\vee z \le y\vee z$ for all $x,y,z \in H$.

{\rm (iv)} The operation $\vee$ is a join operation on $H$.
\end{prop}

\begin{defn}\label{1.55}{\rm \cite[Definition 2.7]{ABAK1}}
A hoop  algebra $H$ is called $\vee$-hoop  if $\vee$ is a join operation on $H$.
\end{defn}

\begin{prop}\label{1.8}{\rm (see \cite[Proposition 2.8 and Lemma 2.9]{GeLePr})}
Let $H$ be a hoop algebra, and let $x,y,z\in H$ and $I$ be an arbitrary set. If there exist  arbitrary joins then

{\rm (i)}    $(x\vee y)\rightarrow z = (x \rightarrow z)\wedge (y \rightarrow z)$,   

{\rm (ii)}   $x\odot (\bigvee_{i\in I} y_i)= \bigvee_{i\in I} (x \odot y_i)$,

{\rm (iii)}  $x\wedge (\bigvee_{i\in I} y_i) = \bigvee_{i\in I} (x \wedge y_i).$ 
\end{prop}

\begin{prop}\label{1.9}{\rm \cite[Proposition 2.15]{GeLePr}}
Let $H$ be a $\vee$-hoop algebra and $x, y, z, t \in H$. Then:

{\rm (i)} If $x\vee y = 1$, then $x\odot y = x\wedge y$.

{\rm (ii)} If $x\vee y = 1$ and $x\le z$, $y\le t$, then $z\vee t = 1$.

{\rm (iii)} If $x\vee y = 1$, then $x^n \vee y^n = 1$ for every $n\in \mathbb{N}$.
\end{prop}

\begin{prop}\label{1.5}{\rm\cite{Bos1,Bos2} }
	Let $H$ be a bounded hoop  algebra. Then for every $x,y \in H$, we have:
	
	{\rm (i)}    $x\le x''$, $x\odot x'= 0$, $x'''= x'$, $x'\le x \rightarrow y$, 
	
	{\rm (ii)}  
	If $x\le y $, then $y'\le x'$, 
	
	{\rm (iii)}  $x\rightarrow y' = y \rightarrow x' = (x\odot y)'$,
	
	{\rm (iv)}   $x\odot y = 0$ if and only if $x\le y'$.
\end{prop}

\begin{lem}\label{1.57}
	Let $H$ be a bounded hoop algebra. Then 
	$x'' \odot y'' \le (x \odot y)''$, for every $x,y \in H$. 
\end{lem}

\begin{proof}
Let $x,y \in H$. By  Proposition~\ref{1.5}(i), $x \odot y \le (x \odot y)''$. Now according to  Theorem~\ref{1.4}(ii), 	$x \le y \rightarrow (x \odot y)''$. 
	Therefore, by items (i) and (iii) of Proposition~\ref{1.5},
	\[
	x \le y \rightarrow (x \odot y)'' = (x \odot y)' \rightarrow y' = (x \odot y)' \rightarrow y''' = y'' \rightarrow (x \odot y)''
	\] 
Now, in view of  Theorem~\ref{1.4}(ii), $x \odot y'' \le (x \odot y)''$.  In a similar way we obtain $x'' \odot y'' \le (x \odot y)''$.	
\end{proof}

\begin{prop}\label{1.60}{\rm \cite{GeLePr}}
	Let $H$ be a bounded hoop  algebra. Then $H$ has $(DNP)$ if and only if $(x\rightarrow y)\rightarrow y = (y \rightarrow x)\rightarrow x,$ for every $x,y \in H$.
\end{prop}

\begin{rmk}\label{1.63}
Let $H$ be a bounded $\vee$-hoop  algebra with $(DNP)$. Then, according to Definition~\ref{1.55} and Proposition~\ref{1.60}, we have $x \vee y = (x\rightarrow y )\rightarrow y = (y \rightarrow x)\rightarrow x.$ 
\end{rmk}

\begin{prop}\label{1.61}
	Let $H$ be a bounded hoop  algebra with $(DNP)$. Then for every $x,y \in H$,
	
	{\rm (i)}  $x \rightarrow y = y'\rightarrow x' $,
	
	{\rm (ii)} $x \rightarrow y = (x \odot y')\rightarrow 0 $,
	
	{\rm (iii)} 	If $H$ is a $\vee$-hoop  algebra, then 	$x \rightarrow (x \odot y) = x' \vee y$,
	
	{\rm (iv)} 
	If $H$ is a  $\vee$-hoop  algebra, then $(x \wedge y)\rightarrow 0 = x' \vee y'$.
\end{prop}

\begin{proof}
(i) We have $y'\rightarrow x' = y' \rightarrow (x \rightarrow 0) = (y'\odot x)\rightarrow 0 = (y'\odot x)'$
	Therefore, according to   Proposition~\ref{1.5}(iii), $y'\rightarrow x' =  (y'\odot x)' = x \rightarrow y'' = x \rightarrow y.$\\
(ii) By   Proposition~\ref{1.5}(iii), it is easily proven.\\
(iii) According to the Remark
\ref{1.63}
and item (i), we have: 
	\begin{align*}
		x' \vee y & = (y \rightarrow x') \rightarrow x' \\ & =  (y \rightarrow (x\rightarrow 0 )) \rightarrow (x \rightarrow 0) \\ & =  ((x \odot y)  \rightarrow 0) \rightarrow (x \rightarrow 0) \\& = x \rightarrow (x \odot y) 
	\end{align*}
(iv) 	According to   Proposition~\ref{1.8}(i), we have: 
	\begin{align*}
		(x' \vee y')' & = (x' \vee y') \rightarrow 0 \\ & = (x'\rightarrow 0) \wedge (y' \rightarrow 0)  \\& = ((x \rightarrow 0 )\rightarrow 0) \wedge ((y \rightarrow 0) \rightarrow 0)  \\ & = x'' \wedge y'' \\ & = x \wedge y  
	\end{align*}
Therefore, 	$ x' \vee y' = (x \wedge y)'$.
\end{proof}

\begin{defn}\label{1.17}{\rm \cite[Definition 1.12]{BlFe2}}
A \textit{Wajsberg hoop}  is a hoop  algebra that satisfies the following condition of antipode (T):
\[ (T)\;\;\;\; (x\rightarrow y)\rightarrow y = (y\rightarrow x)\rightarrow x, \quad \forall x,y \in H \]
Any hoop which satisfies (T) is in fact a lattice, and the operation $\vee$ for every $x,y \in H$ is equal to $x\vee y = (x \rightarrow y )\rightarrow y$.
\end{defn}

\begin{cor}\label{1.18}
Every bounded hoop  algebra with $(DNP)$, is a  Wajsberg hoop  and conversely. 
\end{cor}

\begin{proof}
It is a direct consequence of  Proposition~\ref{1.60} and Definition~\ref{1.17}. 
\end{proof}

\begin{defn}\label{2.1}\cite{BlFe2}
Let $H$ be a hoop  algebra. A subset $F$ of $H$ is called a \textit{filter} of $H$ if it satisfies the following conditions:

(F1) $1\in F$,

(F2) $x\odot y \in F$ for all $x,y\in F$,

(F3) if $x\le y$ and $x\in F$, then $y\in F$ for all $x,y\in H$.
\end{defn}
Let $H$ be a hoop  algebra and $\emptyset \neq X \subseteq H$. The intersection of all filters of $H$ containing $X$ is denoted by $\langle X \rangle$ or $F_g(X)$ and is equal to:
\begin{align*}
	\langle X \rangle &= \Big\{a\in H: \exists   n\in \mathbb{N}, x_1, x_2, \cdots , x_n\in X,\quad x_1 \odot x_2 \odot \cdots \odot x_n \le a\Big\}\\
	&= \Big\{a\in H: \exists   n\in \mathbb{N}, x_1, x_2, \cdots , x_n\in X, x_1 \rightarrow( x_2 \rightarrow( \cdots \rightarrow (x_n \rightarrow a) \cdots )) = 1 \Big\}.
\end{align*}
In particular, for each member $x\in H$ we have:
\begin{align*}
	\langle x\rangle=\{a\in H: \exists   n\in \mathbb{N}, x^n\le a \} = \{a\in H: \exists   n\in \mathbb{N}, x^n \rightarrow a = 1 \}.
\end{align*}
A filter $F$ of $H$ is called  proper  if $F \neq H$. It is clear that $H$ and $\{1\}$ are trivial filters of $H$. It can be seen that if $H$ is a bounded hoop  algebra, then a filter is proper if and only if it does not contain $0$. We denote the set of all filters of the hoop  algebra $H$ by $\mathcal{F}(H)$.

Recall {\cite{Kon}} that given a  filter  $F$  of a hoop algebra $H$, one can define a congruence $\equiv_F$ on  $H$  by
	\[
	x \equiv_F y \Longleftrightarrow x \rightarrow y \quad \textrm{and} \quad y \rightarrow x \in F.
	\]
We will denote the class of an element $x\in H$ by $\frac{x}{F}$ or just by $[x]$. Since the class of all hoops forms a variety, any quotient structure 
	\[\frac{H}{F}= \{\frac{x}{F}:x \in H \}\]
 of $H$ by $\equiv_F$ is also a hoop by the following definition:
\[\frac{x}{F} \odot \frac{y}{F} := \frac{x \odot y}{F},  \quad\frac{x}{F} \rightarrow \frac{y}{F} := \frac{x \rightarrow y}{F}, \quad \textrm{and} \quad 1 := \frac{1}{F}\]
for all $\frac{x}{F} , \frac{y}{F} \in \frac{H}{F}$. In the hoop algebra $(\frac{H}{F};\odot, \rightarrow, 1_F)$,  we also have 
$\frac{x}{F} \le \frac{y}{F}$ if and only if $x \rightarrow y \in F.$

\begin{defn}\label{5.3}{\cite{ABAK1}}
Let $H$ be a $\vee$-hoop  algebra.  A proper filter $P$ of $H$  is  prime  if for every $x,y\in H$, $x\vee y \in P$ implies $x \in P$ or $y \in P$. A proper filter $M$ of the hoop  algebra $H$ is called a \textit{maximal} filter if it is not properly contained in any other filter.
\end{defn}

Let $L$ be a lattice with 0 and 1. For $a\in L$, we say  $b\in L$ is a \emph{complement} of $a$ if $a \vee b = 1$ and $a \wedge b = 0$. If $a$ has a
unique complement, we denote this complement by $a^*$.  The set of all complemented elements in $H$ is denoted by $B(H)$ and is called the Boolean center of $H$.

\begin{lem}\label{4.26}{\rm  \cite[Lemma 1.13]{DPi}}
Let $H$ be a bounded hoop  algebra, $x\in B(H)$, and $x^*$ be the complement of $x$. Then, $x'= x^*$ and $x''=x$.
\end{lem}

\begin{lem}\label{4.22}{\rm \cite[Proposition 1.17]{DPi}}
Let $H$ be a bounded hoop  algebra and $x\in B(H)$. In this case, the following statements hold:

{\rm (i)} $x^2=x$ and $x''=x,$

{\rm (ii)} $x^2=x$ and $x'\rightarrow x= x,$

{\rm (iii)} For all $a\in H,~(x\rightarrow a) \rightarrow x = x,$

{\rm (iv)}  $x'\wedge x = 0.$
\end{lem}

\begin{cor}\label{4.25}
For any bounded hoop  algebra $H$, $B(H)= \id(H)\cap \reg(H)$.
\end{cor}

\begin{defn}\label{2.15}{\cite{GeLePr}}
We call a hoop  algebra $H$ \textit{cancellative} if the monoid $(H;\odot,1)$ is cancellative.
\end{defn}

\begin{prop}\label{2.16} {\rm \cite[Proposition 4.1]{GeLePr}}
A hoop  algebra $H$ is cancellative  if and only if the following identity holds:
$${\rm (c)}\quad y\rightarrow (x\odot y)=x$$
\end{prop}

\begin{prop}\label{2.17}{\rm \cite[Proposition 4.2]{GeLePr}}
Let $H$ be a cancellative hoop  algebra. Then, for any $x,y,z\in H$:

{\rm (i)} $z\rightarrow x = (z\odot y )\rightarrow (x\odot y)$,

{\rm (ii)} $x\odot y\le z \odot y$ if and only if $x \le z$.
\end{prop}

\begin{defn}\label{2.18}{\cite{GeLePr}}
We call a hoop algebra $H$ \textit{basic} if the following condition holds for every $x,y,z\in H$:
$${\rm (BH)}\;\;\;\;(x\rightarrow y)\rightarrow z \le ((y\rightarrow x)\rightarrow z) \rightarrow z$$
\end{defn}

\begin{lem}\label{2.19}{\rm \cite{GeLePr}}
Assume $H$ is a basic hoop algebra. Then, for every $x,y,z\in H$:

{\rm (i)} $(x\rightarrow y)\vee(y \rightarrow x) = 1$, 

{\rm (ii)} $x\rightarrow y = (x \vee y) \rightarrow y$,
	
{\rm (iii)} 
$(x\vee y) \rightarrow z = (x \rightarrow z) \wedge (y \rightarrow z)$,
	
{\rm (iv)} $x\odot (y \wedge z) = (x \odot y) \wedge (x \odot z)$,
	
{\rm (v)} $(x\rightarrow y) \rightarrow (y \rightarrow x) = y \rightarrow x$. 
\end{lem}

\begin{prop}\label{2.66}
Let $H$ be a basic hoop algebra. If $H=Id(H)$, Then $x \odot (y \odot z) \le (x \wedge y) \odot (x \wedge z)$ for any $x,y,z\in H$. 
\end{prop}

\begin{proof}
According to   Theorem~\ref{1.4}(iii), we have	$x \wedge (y \odot z) \le x \wedge y$ and $x \wedge (y \odot z) \le x \wedge z$.
By multiplying two inequality we have 	$(x \wedge (y \odot z)) \odot (x \wedge (y \odot z)) \le (x \wedge y) \odot (x \wedge z)$.
	On the other hand, by  Lemma~\ref{2.19}(iv), we have
\begin{align*}
	(x \wedge (y \odot z)) \odot (x \wedge (y \odot z)) & = [x \odot (x \wedge (y \odot z)) ] \wedge [(y \odot z) \odot (x \wedge (y \odot z) ]
	\\ & = (x\odot x) \wedge (x \odot (y \odot z)) \wedge ((y \odot z) \odot x) \wedge ((y \odot z) \odot (y \odot z))  \\ & = x \wedge (x \odot (y \odot z)) \wedge (y \odot z)  \\ & = x \odot (y \odot z)
\end{align*}
Therefore, the assertion is concluded.
\end{proof}

\begin{prop}\label{2.20}{\rm \cite[Proposition 4.9]{GeLePr}}
Any Wajsberg hoop is a basic hoop. 
\end{prop}

\begin{cor}\label{2.21}
Every bounded hoop algebra with $(DNP)$ is basic hoop.
\end{cor}

\begin{proof}
By using Corollary \ref{1.18} and Proposition \ref{2.20}, the result follows.
\end{proof}

\begin{rmk}\label{2.67} 
	It is substantial to mention that, in view of Proposition \ref{2.20} and Corollary
	\ref{2.21}, the propositions which  hold for basic hoop algebras also valid for Wajsberg hoops and bounded hoops with $(DNP)$.
\end{rmk}

\begin{defn}\label{2.30}{\rm \cite[Definition 1.1]{DPi}}
A \textit {residuated lattice} is an algebra $\mathbf{L} = (L;\wedge,\vee,\odot,\rightarrow,0,1) $ of type (2,2,2,2,0,0) equipped with an order ($\le$) satisfying the following:
	\begin{itemize}
		\item[(RL1)] $(L;\wedge,\vee,0,1)$ is a bounded lattice,
		
		\item[(RL2)] $(L;\odot,1)$ is a commutative ordered monoid,
		
		\item[(RL3)] $\odot$ and $\rightarrow$	form an adjoint pair, i.e.,	$x\odot y \le z$ if and only if $x\le y\rightarrow z$.
	\end{itemize}
\end{defn}

\begin{defn}\label{2.31}
	{\cite{Hol}}
Let $\mathbf{L} = (L;\wedge,\vee,\odot,\rightarrow,0,1)$ be a residuated lattice. We call $\mathbf{L}$ \textit{divisible} if $x \wedge y = x \odot (x \rightarrow y)$ for every $x,y\in L$.
\end{defn}

\begin{prop}\label{2.32}
{\rm\cite{BlFe2}}
Let $H$ be a divisible residuated lattice. Then $H$ is a bounded hoop which is a $\vee$-semilattice. 
\end{prop}

\section{Hoop  algebras and square roots}
In 1995, H\"{o}hle \cite{Hol} defined the square root as a unary operation on an $l$-groupoid, which is commutative, residuated, and integral, and investigated properties in this type of algebras. He then studied the influence of the square root on MV-algebras and ultimately proposed a new classification for these algebras based on the square root. In 1999, Ambrosio \cite{Amb} presented other properties for strict MV-algebras. In 2015, Chen and Dudek \cite{ChDu} studied the representation and properties of square roots in pseudo-MV-algebras. Finally, in 2023, A. Dvure\v censkij and Zahiri \cite{DvZa1} presented representations and other properties of the square root on pseudo-MV-algebras. This section introduces the definitions of the square root and the $n$-th root in hoop  algebras, and provides properties of these operations.
\begin{defn}\label{3.1}{\rm \cite{Hol}}
Let $H$ be a hoop  algebra. A mapping $s:H \rightarrow H$ is called a \textit{square root} on $H$ if it satisfies the following conditions:

(S1) For every $x\in H$, $s(x)\odot s(x)=x$,

(S2) For all $x,y\in H$, if $y\odot y \le x$, then $y\le s(x)$.
\end{defn}

\begin{exm}\label{3.2}
a) Let $I=\left [ 0,1 \right ] $ be the unit closed interval of real numbers, and let $a\in I$. Let $G_a = \{a^t : t\in \mathbb{Q^+} \cup \{0\}\}$ where $a^0=1.$
We define the multiplication of two elements as $a^t \odot a^s = a^{t+s}$ and the operation $\rightarrow$ is given by $a^t \rightarrow a^s = a^{\max(s-t,0)}.$ 
It can be easily verified that $(G_a;\odot,\rightarrow,1)$ is a hoop algebra. We define the unary operation $s:G_a \rightarrow G_a $ by $s(a^\frac{m}{n})= a^\frac{m}{2n},$ for every $a^\frac{m}{n}$. We check the conditions of Definition \ref{3.1}. We have:
\[ s(a^\frac{m}{n})\odot s(a^\frac{m}{n}) = a^\frac{m}{2n}\cdot a^\frac{m}{2n} = a^\frac{m}{n} \]
Thus, condition (S1) holds. Additionally,
\[ a^t \odot a^t\le a^\frac{m}{n} \Rightarrow a^{2t} \le a^\frac{m}{n} \Rightarrow a^t\le a^\frac{m}{2n} \Rightarrow a^t \le s(a^\frac{m}{n}). \]
Therefore, condition (S2) is also satisfied. Hence, the operation $s$ is a square root on $G_a$.\\
b) Let $I=\left [ 0,1 \right ]$ be the unit closed interval of real numbers. We define the operations $\odot$ and $\rightarrow$ on $I$ as follows. For any $x,y \in I$,
\[ x\odot y = \min{(x,y)} \quad \text{and} \quad x\rightarrow y =
\begin{cases} 
1 & \text{if } x \le y \\ 
y & \text{if } x > y 
\end{cases}.
\]
The structure $\textbf{Goa}:=(I;\odot,\rightarrow,1)$ defined with these operations  is a hoop  algebra. This algebraic structure is known as the G\"{o}del algebra (see~ \cite{DPi,DBDPAJ}).

In $\textbf{Goa}$, we define the unary operation $s: I \rightarrow I$ as $s(x)= x$ for every $x\in I$. It is evident that the conditions of Definition \ref{3.1} hold. Therefore,  $s$ is a square root on $\textbf{Goa}$.\\
c) Let $I=\left [ 0,1 \right ]$ be the unit closed interval of real numbers. We define the operations $\odot$ and $\rightarrow$ on $I$ as follows. For any $x,y \in I$,
\[ x\odot y =x \cdot y \quad \text{and} \quad x \rightarrow y =
\begin{cases} 
1 & \text{if } x \le y \\ 
\frac{y}{x} & \text{if } x > y 
\end{cases}.
\]
The structure $\textbf{Pra} := (I;\odot,\rightarrow,1)$ defined with these operations  is a hoop  algebra (see~\cite{DPi,DBDPAJ}). This algebraic structure is known as the product algebra.\\
In the algebra $\textbf{Pra}$, we define the unary operation $s: I \rightarrow I $ as
\[s(x)= \sqrt{x},\;\; \forall x\in I\]
In this case, we have $s(x)\odot s(x)=\sqrt{x} \cdot \sqrt{x} = x$. Thus, the condition (S1) of Definition \ref{3.1} is satisfied. Also,
\[y\odot y \le x \Rightarrow y^2\le x \Rightarrow y\le \sqrt{x} \Rightarrow y \le s(x).\]
Therefore, condition (S2) is also met. Hence, the operation $s$ is a square root on $\textbf{Pra}$.\\
d) Let $(B;\vee,\wedge,',0,1)$ be a Boolean algebra. We define the following operations on $B$:
\[ x\odot y =x\wedge y,\quad x\rightarrow y = x'\vee y\]
for all $x,y \in B$. Then, $(B;\odot,\rightarrow,0,1)$ is a bounded hoop  algebra. In $B$, we define the unary operation $s:B \rightarrow B$ as $s(x)= x$ for all $x\in B$.
In this case, we have $s(x)\odot s(x)= x\wedge x = x$. Also, $ y\odot y \le x \Rightarrow y \wedge y \le x \Rightarrow y \le x \Rightarrow y \le s(x).$ 
Therefore,  $s$ is a square root on $B$.
\end{exm}

\begin{thm}\label{3.3}
Let $H$ be a hoop  algebra and $s:H \rightarrow H$ be a square root on $H$. Then for every $x\in H$
\begin{itemize}
	\item[{\rm (i)}] $x\leq s(x)$. In particular, if  $H$ is bounded and $s(x)=0$, then $x=0$.

\item[{\rm (ii)}] $s(1)=1$.

\item[{\rm (iii)}] $x=s(x)\odot (s(x)\rightarrow x)$.
\end{itemize}
\end{thm}
\begin{proof}
(i) By Theorem \ref{1.4}, for every $x\in H$, we have $x=s(x) \odot s(x) \le s(x)$.\\
(ii) By Theorem \ref{1.4}(vii), and item (i), it is clear.\\
(iii) By  (i), for every $x\in H$, we have $x \rightarrow s(x)=1$. Then, by  condition $H3$ of Definition~\ref{1.1}, we get $x=s(x) \odot (s(x) \rightarrow x)$.
\end{proof}

\begin{prop}\label{3.30}
Let $H$ be a hoop  algebra and $s:H \rightarrow H$ be a square root on $H$. Then, for every $x\in H$, $x^2=x$ if and only if $s(x)=x$.
\end{prop}

\begin{proof}
Assume $x^2=x$ for every $x\in H$. In this case, $x=s(x) \odot s(x) = s(x)$. Conversely, assume $s(x)=x$ for every $x\in H$. Therefore,
$x=s(x) \odot s(x)= x \odot x=x^2$.

\end{proof}

\begin{thm}\label{3.31}
Let $H$ be a hoop  algebra and $s:H \rightarrow H$ be a square root on $H$. Then the square root is unique.
\end{thm}

\begin{proof}
Assume $s:H \rightarrow H$ and $r:H \rightarrow H$ are two square roots on $H$. According to Definition \ref{3.1}, for every $x\in H$, we have $s(x)\odot s(x)\le x \Rightarrow s(x)\le r(x)$. Similarly, $r(x)\le s(x)$. Therefore, $r(x)= s(x)$.
\end{proof}

\begin{thm}\label{3.80}
	Let $H$ be a hoop  algebra and $s:H \rightarrow H$ be a square root on $H$. Then the square root is one to one.
\end{thm}

\begin{proof}
Let  $x,y\in H$ be such that $s(x)=s(y)$. Then, we have $s(x)\odot s(x)=s(y)\odot s(y)$ and so $x=y.$
\end{proof}

\begin{cor}\label{3.5}
Let $H$ be a hoop algebra and $s:H \rightarrow H$ be a square root on $H$. Then, for every $x\in H$, $x=1$ if and only if $s(x)=1$.
\end{cor}

\begin{proof}
If $x=1$, then by Theorem \ref{3.3}, $s(x)=1$. Conversely, if $s(x)=1$, then by Theorem~\ref{3.3}(ii), $s(x)=s(1)$. Therefore, by Theorem~\ref{3.80}, $x=1$.
\end{proof}

\begin{prop}\label{3.29}
	Let $H$ and $G$ be two hoop  algebras. Suppose $s_1: H \to H$ and $s_2: G \rightarrow G$ are  square roots on $H$ and $G$, respectively. Then, the function 
	$s: H \times G \rightarrow H \times G$ defined by $s(x,y) = (s_1(x), s_2(y))$ is a square root on $H \times G$.
\end{prop}
\begin{proof}
We examine the validity of  conditions of Definition \ref{3.1}. For every $x_1 \in H$ and $x_2 \in G$, by Proposition \ref{1.15}, we have
	\begin{align*}
		s(x_1,x_2) \odot s(x_1,x_2) &= (s_1(x_1), s_2(x_2)) \odot (s_1(x_1), s_2(x_2))\\
		&= (s_1(x_1) \odot s_1(x_1), s_2(x_2) \odot s_2(x_2))\\
		&= (x_1, x_2)
	\end{align*}
Thus, condition $S1$ holds. To prove condition $S2$, suppose $y_1 \in H$ and $y_2 \in G$ such that $(y_1, y_2) \odot (y_1, y_2) \le (x_1, x_2)$. Therefore, we have $y_1 \odot y_1 \le x_1$ and $y_2 \odot y_2 \le x_2$. Since $s_1$ and $s_2$ are the square roots on $H$ and $G$ respectively, we have $y_1 \le s_1(x_1)$ and $y_2 \le s_2(x_2)$. Then, $(y_1, y_2) \le (s_1(x_1), s_2(x_2))$. Hence, condition $S2$ is also satisfied.
\end{proof}

\begin{lem}\label{3.7}
Let $H$ be a hoop  algebra and $a\in \id(H)$. We define $H[a]=\{x\in H : a \le x\}$. Then, $H[a]$ is a subalgebra and a filter of $H$.
\end{lem}

\begin{proof}
We let the operations $\odot$ and $\rightarrow$ on $H[a]$ to be the same as the operations $\odot$ and $\rightarrow$ in $H$. It is clear that $1\in H[a]$. Let $x,y\in H[a]$. Then, due to $a \le x$ and $a \le y$, and by Theorem \ref{1.4}(xiii), we have $a=a \odot a \le x \odot y$. Thus, $x \odot y \in H[a]$. Therefore, $H[a]$ is closed under multiplication. Also, according to items (v) and (xiv) of Theorem~\ref{1.4}, we have $a \le x \rightarrow a \le x \rightarrow y$.  Hence, $H[a]$ is a subalgebra of $H$.

Considering the conditions F1 and F2 of Definition \ref{2.1} from the above explanations, and with the fulfillment of condition F3, $H[a]$ is a filter of $H$.
\end{proof}

\begin{thm}\label{3.8}
Let $H$ be a hoop  algebra and $s: H \rightarrow H$ be a square root on $H$. Then, for every $x, y \in H$:
\begin{itemize}
\item[{\rm (i)}] If $x\leq y$, then $s(x)\le s(y)$.

\item[{\rm (ii)}] $x\le s(x^2)\le s(x)$ and $s(x^2)\odot s(x^2)=s(x)\odot s(x)\odot s(x)\odot s(x)=x^2$.
	
\item[{\rm (iii)}] $x \wedge y \le s(x)\odot s(y)$,
	
\item[{\rm (iv)}] $s(x \wedge y)\le s(x)\wedge s(y)$,
	
\item[{\rm (v)}] $s(x)\odot s(y) \le s(x\odot y)$,
		
\item[{\rm (vi)}] $s(x)\rightarrow s(y) = s(x\rightarrow y)$,
	
\item[{\rm (vii)}] $x^2\le s(x)\odot s(x^2)\le x$,		
	
\item[{\rm (viii)}] If $y\le s(x)$, then  $y^n\le x$, for every $n\in \mathbb{N}$ with $n \ge 2$.

\item[{\rm (ix)}] If $a\in \id(H)$, then $s_a: H[a] \rightarrow H[a]$ induced by $s$ is a square root on $H[a]$.

\item[{\rm (x)}]  $s(x)\le s(x^{n-1})\rightarrow s(x^{n})$, for every $n\in \mathbb{N}$.

\item[{\rm (xi)}] If $x\le y$, then $s(\left[x,y\right])\subseteq \left[s(x),s(y)\right]$.
\end{itemize}	

\end{thm}
\begin{proof}	
(i) We have $s(x)\odot s(x)=x\le y$, so by the condition (S2) of Definition~\ref{3.1}, $s(x)\le s(y)$.\\
(ii) We have $x\odot x\le x\odot x$, therefore by the condition $(S2)$ of Definition~\ref{3.1},   $x\le s(x\odot x)=s(x^2)$ and by (i), $s(x^2)\le s(x)$.
Moreover, by (S1), $s(x\odot x)\odot s(x\odot x) = x\odot x = s(x)\odot s(x)\odot s(x)\odot s(x)$.\\
(iii) Using part (i) and the compatibility of the partial order with multiplication, we have $x\wedge y = s(x\wedge y)\odot s(x\wedge y) \le s(x)\odot s(y)$.\\
(iv) The proof is straightforward.\\
(v) We have $(s(x)\odot s(y)) \odot (s(x)\odot s(y)) = (s(x)\odot s(x))\odot (s(y)\odot s(y)) = x \odot y$. Then, by the condition (S2) of Definition \ref{3.1}, $s(x)\odot s(y)\le s(x \odot y)$.\\
(vi) We have
\begin{align*}
x \odot (s(x)\rightarrow s(y))\odot (s(x)\rightarrow s(y))&= s(x)\odot s(x) \odot (s(x)\rightarrow s(y))\odot (s(x)\rightarrow s(y))\\
&= s(x) \odot (s(x)\wedge s(y))\odot (s(x)\rightarrow s(y))\\
&= (s(x)\wedge s(y))\odot s(x) \odot (s(x)\rightarrow s(y))\\ 
& \le (s(x)\wedge s(y)) \odot s(y) \\
& \le s(y) \odot s(y)\\
&= y 
\end{align*}
Therefore, by Theorem~\ref{1.4}(ii), $(s(x)\rightarrow s(y))\odot (s(x)\rightarrow s(y))\le x\rightarrow y$. From Definition $\ref{3.1}$, we get $s(x)\rightarrow s(y) \le s(x\rightarrow y)$. On the other hand, by the part (v), we have
\[
s(x)\odot s(x\rightarrow y)\le s(x\odot(x\rightarrow y))=s(x\wedge y)\le s(y)
\]
Therefore, according to Theorem \ref{1.4}(ii), $s(x\rightarrow y)\le s(x)\rightarrow s(y)$. Hence, $s(x\rightarrow y)= s(x)\rightarrow s(y)$.\\
(vii) According to Theorem \ref{3.3}(i), $x\le s(x)$. Also, by part (ii), $x\le s(x^2)$. Therefore, $x^2\le s(x)\odot s(x^2)$. Furthermore, we have $s(x^2)\le s(x)$. Therefore,
\[
s(x)\odot s(x^2)\le s(x)\odot s(x)=x  
\]
(viii) Suppose $y\le s(x)$. We have $y \odot y \le s(x)\odot s(x)=x $. Therefore, the assertion follows by induction.\\ 
(ix) According to Lemma $\ref{3.7}$, the proof is easily done.\\
(x) We have $x^n=x^{n-1}\odot x$. By part (v), $s(x^n)\ge s(x^{n-1}) \odot s(x)$. Therefore, using Theorem $\ref{1.4}$(ii), $s(x)\le s(x^{n-1})\rightarrow s(x^{n})$.\\
(xi) The proof is straightforward according to part (i).
\end{proof}

\begin{cor}\label{3.9}
Let $H$ be a hoop  algebra,  $x\in H$ and $n \in \mathbb{N}$. Then 
\begin{itemize}
\item[{\rm (i)}] $s(x^{n})^2 = s(x)^{2n}= x^n$,
\item[{\rm (ii)}] $s(x)^{n}\le s(x^n)$.
\end{itemize}
\end{cor}

\begin{prop}\label{3.10} 
Let $H$ be a hoop algebra and $s:H \rightarrow H$ be a square root on $H$. For every $x,y\in H$ one has
\begin{itemize}
\item[{\rm (i)}] If $x = s(x)$ and $y = s(y)$, then $s(x)\odot s(y) = s(x\odot y)$.
    
\item[{\rm (ii)}] If $x \odot y = x \wedge y$, then $s(x)\odot s(y) = s(x\odot y)$.

\item[{\rm (iii)}] If $s(x)\odot s(y) = s(x\odot y)$, then $ s(x \wedge y) = s(x) \wedge s(y).$

\item[{\rm (iv)}] If $s(x)\odot s(y) = s(x\odot y)$, then $x= s(x^2)$. In particular, $s(0)=0$.
\end{itemize}
\end{prop}
\begin{proof}
(i) The proof is straightforward by Proposition~\ref{3.30}.\\
(ii) By Theorem \ref{3.8}(v), we know that 
\begin{equation}\label{s(x)s(y)< s(xy)}
s(x)\odot s(y) \le s(x\odot y).
\end{equation} 
On the other hand, by part (i) of the same Theorem, we have $s(x\odot y) \le s(x)$ and $s(x\odot y) \le s(y)$. Therefore, $s(x\odot y) \le s(x) \wedge s(y)$. According to the assumption of the proposition, we have $s(x) \wedge s(y) = s(x)\odot s(y)$. Then, 
\begin{equation}\label{converse eq.}
s(x\odot y) \le s(x)\odot s(y).
\end{equation}
 Thus, from (\ref{s(x)s(y)< s(xy)}) and (\ref{converse eq.}), we get the result.\\
(iii) According to Theorem~\ref{1.4}(i), we know $x\wedge y = x \odot (x\rightarrow y)$. Now, according to the hypothesis of  Proposition and   Theorem~\ref{3.8}(vi), we have
\begin{eqnarray*}
	s(x\wedge y) &=& s(x \odot (x\rightarrow y)) \\ &=&  
	s(x) \odot s(x\rightarrow y) \\ &=&  
	s(x) \odot (s(x) \rightarrow s(y)) \\ &=&
	s(x) \wedge s(y) 
\end{eqnarray*}  
(iv) Suppose $s(x)\odot s(y) = s(x\odot y)$ and $x=y$. In this case, we have $x=s(x)\odot s(x)= s(x\odot x)=s(x^2)$. Also,  we have $0 =s(0)\odot s(0)= s(0\odot 0)=s(0)$, since, according to Theorem~\ref{1.4}(iii), $0 \odot 0 = 0$.
\end{proof}

\begin{exm}\label{3.11}
a) In Example \ref{3.2}(a), we have $ s(a^\frac{m}{n})= a^\frac{m}{2n}.$ Now, if $x=a^\frac{m}{n}$ and $y=a^\frac{u}{v}$, then 
\begin{eqnarray*}
s(x)\odot s(y) &=& s(a^\frac{m}{n})\odot s(a^\frac{u}{v}) \\ 
&=&  a^{\frac{m}{2n}+\frac{u}{2v} }\\
 &=& a^{\frac{mv+nu}{2nv}} \\ 
 &=& s(a^{\frac{mv+nu}{nv}})\\ 
 &=& s(a^{\frac{m}{n}+\frac{u}{v}}) \\ 
 &=& s(a^{\frac{m}{n}} \cdot a^{\frac{u}{v}})\\ 
 &=& s(x \odot y) 
\end{eqnarray*}
Therefore, the relation $s(x)\odot s(y) = s(x\odot y)$ holds for every $x,y \in G_a$.\\
b) In the G\"{o}del hoop $\textbf{Goa}$, in Example \ref{3.2}(b), let  $s:I \rightarrow I$ be the identity function. Then clearly  the relation $s(x)\odot s(y) = s(x\odot y)$ holds for every $x,y \in I$.
\end{exm}
Similarly, this   holds also for Examples \ref{3.2}(c) and \ref{3.2}(d).
 
\begin{thm}\label{3.12} 
Let $H$ be a basic hoop  algebra and $s:H \rightarrow H$ be a square root on $H$. Then, for every $x,y\in H$:
\begin{itemize}
    \item[{\rm (i)}] $s(x\wedge y)= s(x)\wedge s(y)$,
	
    \item[{\rm (ii)}] If $y\le s(x)\odot s(y)$, then $y\le x$.
 \end{itemize}
\end{thm}
\begin{proof}
(i) By Theorem \ref{3.8}(iv), we have $s(x\wedge y)\le s(x)\wedge s(y)$. On the other hand, based on Lemma \ref{2.19}(iv),
\begin{align*}	
	(s(x)\wedge s(y)) \odot (s(x)\wedge s(y)) &= (s(x)\odot s(x))\wedge (s(x)\odot s(y)) \wedge (s(y)\odot s(x))\wedge (s(y)\odot s(y)) \\
&\le x \wedge y 
\end{align*}
Therefore, according to the condition (S2) in Definition \ref{3.1}, we have $s(x)\wedge s(y)\le s(x \wedge y)$. Thus, the equality $s(x\wedge y) = s(x)\wedge s(y)$ holds.\\
(ii) Let $x,y \in H$ such that $y \le s(x)\odot s(y)$. Then, we have
\begin{eqnarray*}
y&=& y \wedge (s(x)\odot s(y)) \\
 &=& (s(y)\odot s(y)) \wedge (s(x)\odot s(y)) \\ 
 &=& s(y) \odot (s(x)\wedge s(y))\\ 
 &=& s(y) \odot (s(y)\odot (s(y) \rightarrow s(x))) \\ 
 &=& y \odot (s(y) \rightarrow s(x)) \\ 
 &\le& s(x)\odot s(y) \odot (s(y) \rightarrow s(x)) \\
 &\le& s(x)\odot s(x)\\ 
&=& x 
\end{eqnarray*}	
\end{proof}

\begin{thm}\label{3.13}
Let $H$ be a bounded hoop  algebra and $s:H \rightarrow H$ be a square root on $H$. Then, for every $x,y\in H$:
\begin{itemize}
	\item[{\rm (i)}] $s(x)\odot s(0)\le x$,
	
	\item[{\rm (ii)}] If $x\in \id(H)$ such that $x\le s(0)$, then $x=0$.
	
	\item[{\rm (iii)}] $x\wedge x'\le s(0)$,
	
    \item[{\rm (iv)}] $(s(x)\odot s(0)')' \rightarrow s(y)\le s(x'\rightarrow y)$,
	
	\item[{\rm (v)}] $s(x)\le x'\rightarrow s(0)$,
	
	\item[{\rm (vi)}] $s(x)\le s(0) \rightarrow x$,
	
	\item[{\rm (vii)}] $s(x') \otimes s(0) = s(x')$ where $x \otimes y = (x\rightarrow s(0))\rightarrow y$,
	
	\item[{\rm (viii)}] $s(x')\odot s(x) \le s(0)$,
	
	\item[{\rm (ix)}] $s(x')'\le s(x'')$,
	
	\item[{\rm (x)}] $s(0)= \max \{x\wedge x' : x \in H\}$,
	
	\item[{\rm (xi)}] $s(x) \rightarrow s(x')=s((x^2)')$. 	
	Moreover, if $x=x^2$, then $s(x) \rightarrow s(x') = s(x')$.
	
	\item[{\rm (xii)}] $s(x)'\le s(x')$,
	
	\item[{\rm (xiii)}] 
	If 	 $x \in \reg(H)$, then  $s(x) \in \reg(H)$,	
	
	\item[{\rm (xiv)}] $s(x\odot y) \le ((s(x) \odot s(y))\rightarrow s(0))\rightarrow s(0)$. In particular, if $x \in \reg(H)$, then	$s(x^2) = x \vee s(0)$.
	
	\item[{\rm (xv)}] $s(0) \le s(0)'$.	
\end{itemize}	 
\end{thm}
\begin{proof}
(i) For every $x\in H$, we have $0\le x$. Therefore, by Theorem~\ref{3.8}(i) and Theorem~\ref{1.4}(xiii), we have $s(0)\le s(x) \Rightarrow s(x)\odot s(0)\le s(x)\odot s(x) = x.$\\
(ii) Suppose $x\in I(H)$ is such that $x\le s(0)$. Then $x = x\odot x \le s(0)\odot s(0) = 0.$\\
(iii) By items (i) and (iii) of Theorem \ref{1.4}, we have $(x\wedge x')\odot (x\wedge x') \le x \odot x' = 0.$ Therefore, according to condition (S2) in Definition~\ref{3.1}, $x\wedge x'\le s(0)$.\\
(iv) Using Theorem \ref{1.4}(xi), we get:
\begin{eqnarray*}
s(x)\rightarrow s(0) &\le& (s(0) \rightarrow 0) \rightarrow (s(x)\rightarrow 0) \\
&=& s(0)' \rightarrow (s(x)\rightarrow 0)
\end{eqnarray*}
Now, by Theorem \ref{1.4}(xv), we have $(s(0)' \rightarrow (s(x)\rightarrow 0))\rightarrow s(y) \le (s(x)\rightarrow s(0))\rightarrow s(y).$ On the other hand, by Theorem \ref{3.8}(vi), we have:
\[s(x'\rightarrow y)= s(x')\rightarrow s(y)=(s(x)\rightarrow s(0))\rightarrow s(y).\]
Thus, we conclude $(s(0)'\rightarrow (s(x)\rightarrow 0))\rightarrow s(y) \le s(x'\rightarrow y).$ Now, using proposition~\ref{1.5}(iii), the result is obtained.\\
(v) Using Theorem \ref{3.8}(i,vi) and Theorem \ref{1.4}(ii), we have:
\[x'\le s(x')=s(x\rightarrow 0)=s(x)\rightarrow s(0) \Rightarrow s(x)\le x'\rightarrow s(0).\]
(vi) Since $0 \le x$ for every $x\in H$, we have $s(0)\le s(x)$. Now, we have $s(0)\odot s(x)\le s(x)\odot s(x)=x$. Therefore, by Theorem \ref{1.4}(ii), $s(x)\le s(0)\rightarrow x$.\\
(vii) By Theorem \ref{1.4}(xx) and Theorem \ref{3.8}(vi):
\begin{align*}
s(x')\otimes s(0)&= (s(x')\rightarrow s(0))\rightarrow s(0)\\
&= ((s(x)\rightarrow s(0))\rightarrow s(0))\rightarrow s(0)\\
&= s(x)\rightarrow s(0)= s(x'). 
\end{align*}
(viii) Using Theorem \ref{3.8}(vi), we have
\[s(x')\odot s(x)= s(x)\odot (s(x)\rightarrow s(0))= s(0)\odot (s(0)\rightarrow s(x))\le s(0).\]
(ix) According to Theorem \ref{3.3}(i), for every $x\in H$, $x'\le s(x')$. Now, by Proposition \ref{1.5}(ii), we get $s(x')' \le x''\le s(x'')$.\\
(x) By part (iii), for every $x\in H$, $x\wedge x'\le s(0)$. So, $s(0)$ is an upper bound for the set $\{x\wedge x' : x \in H\}$. Also, since $s(0)\odot s(0)=0$, by part (vi), we have $s(0)\le (s(0))'$. Hence, $s(0)= s(0)\wedge (s(0))' \in \{x\wedge x' : x \in H\}$. Therefore, $s(0)= \max \{x\wedge x' : x\in H\}$.\\
(xi) The proof is straightforward.\\
(xii) We know that $0 \le s(0)$. Therefore, by Theorem \ref{1.4}(xiv) and Theorem \ref{3.8}(vi), for every $x\in H$, we have $s(x)\rightarrow 0 \le s(x)\rightarrow s(0)$ and so $s(x)'\le s(x').$\\
(xiii)
Let  $x \in \reg(H)$. Then $x'' = x$. According to  Proposition \ref{1.5}(i), we know  $s(x) \le s(x)''$. Now by Lemma~\ref{1.57}, we have $s(x)'' \odot s(x)'' \le (s(x) \odot s(x))'' = x'' = x$. Therefore, according to the condition $(S2)$ of Definition \ref{3.1}, $s(x)'' \le s(x)$. Therefore, the results follows.\\
(xiv) We have
\begin{align*}
	s(x \odot y) &\le s((x\odot y)'') & \\ & = s(((x\odot y) \rightarrow 0) \rightarrow 0) \\
	&= s((x \odot y) \rightarrow 0) \rightarrow s(0)\\
	&= s(x \rightarrow (y \rightarrow 0)) \rightarrow s(0) \\
	&= (s(x) \rightarrow s(y \rightarrow 0)) \rightarrow s(0) \\
	&= (s(x) \rightarrow (s(y) \rightarrow s(0))) \rightarrow s(0) \\
	&= ((s(x) \odot s(y)) \rightarrow s(0)) \rightarrow s(0). 
\end{align*}
(xv) Using  Definition~\ref{3.1} and   Theorem~\ref{1.4}(ii), the result is obtained.
\end{proof}

\begin{cor}\label{3.22}
Let $H$ be a bounded hoop  algebra and $s:H \rightarrow H$ be a square root  in $H$. Then $x\in N(H)$ if and only if $s(x)\in N(H)$ for every $x\in H$.
\end{cor}
\begin{proof}
Assume $x\in N(H)$, so there exists an $n \in \mathbb{N}$ such that $x^n=0$. On the other hand, since $s(x)^2 = x$, we have $(s(x)^2)^n = 0$. Therefore, $s(x)\in N(H)$. Conversely, suppose $s(x)\in N(H)$. Thus, there exists an $n \in \mathbb{N}$ such that $s(x)^n=0$. Now, by Theorem \ref{3.3}(i), we have $x^n=0$. Hence, $x\in N(H)$.
\end{proof}

\begin{cor}\label{3.23}
Let $H$ be a bounded hoop  algebra and $s:H \rightarrow H$ be a square root  on $H$. Then the following conditions are equivalent:
\begin{itemize}
	\item[{\rm (i)}] $H$ is locally finite,
	
	\item[{\rm (ii)}] For every $x\in H$, if $x\neq 1$, then $\ord(s(x))< \infty$.
\end{itemize}
\end{cor}
\begin{proof}
(i)$\Rightarrow$(ii) is clear. \\
(ii) $\Rightarrow$ 	(i). Let $x \in H$ and $x\ne 1$. Since $\ord(s(x))<\infty$ by Corollary ~\ref{3.22}, $\ord(x)< \infty$. 	
\end{proof}

\begin{prop}\label{3.24}
Let $H$ be a bounded hoop  algebra and $s:H \rightarrow H$ be a square root in $H$. Then for every $x\in H$, $x\in D(H)$ if and only if $s(x)\in D(H)$.
\end{prop}
\begin{proof}
Assume $x\in D(H)$, so $x'=0$. By Theorem  \ref{3.3}(i) and    Proposition \ref{1.5}(ii), we have $s(x)' \le x'$. Therefore, $s(x)\in D(H)$. Conversely, suppose $s(x)\in D(H)$. Thus, $s(x)' =0$. Now, according to condition (H4) of Definition \ref{1.1}, we have:
\begin{align*}
x' &= x \rightarrow 0 \\
&= (s(x) \odot s(x)) \rightarrow 0 \\
&= s(x) \rightarrow (s(x) \rightarrow 0) \\
&= s(x) \rightarrow s(x)' \\
&= s(x) \rightarrow 0 \\
&= s(x)' \\
&= 0
\end{align*}
Therefore, $x\in D(H)$.  
\end{proof}

\begin{thm}\label{3.14}
Let $H$ be a bounded $\vee$-hoop algebra and $s:H \rightarrow H$ be a square root  in $H$. Then for every $x,y\in H$:
\begin{itemize}
	\item[{\rm (i)}] $x\vee s(0)\le s(x^2)$. If $x= x^2$, then $x \vee s(0) = s(x^2)$,
	
	\item[{\rm (ii)}] $(s(x)\vee s(0))^2 = x$,
	
	\item[{\rm (iii)}] $s(x)\vee s(y)\le s(x\vee y)$. If $H$ is basic, then $s(x)\vee s(y)= s(x\vee y)$,
	
\item[{\rm (iv)}] If $x\vee y =1$, then $s(x^n)\vee s(y^n)= s(x)^n\vee s(y)^n =1,$ for every $n\in \mathbb{N}$.
\end{itemize}	
\end{thm}
\begin{proof}
(i) By Theorems \ref{3.3}(i) and \ref{3.8}(ii), the proof is straightforward.\\
(ii) We know that $s(x)\vee s(0) = s(x)$. Therefore, we have
\[
(s(x)\vee s(0))\odot (s(x)\vee s(0))= s(x) \odot s(x) = x
\]
(iii) Since $x \le x\vee y$, by Theorem \ref{3.8}(i), we have $s(x) \le s(x\vee y)$. Similarly, $s(y) \le s(x\vee y)$. Thus, it follows that $s(x)\vee s(y) \le s(x\vee y)$. If moreover $H$ is basic, then using Theorems \ref{3.8}(iv,vi) and \ref{3.12}(i), we have:
\begin{eqnarray*}
s(x\vee y) &=& s(((x \rightarrow y )\rightarrow y) \wedge ((y \rightarrow x) \rightarrow x)) \\
&=& s((x \rightarrow y )\rightarrow y) \wedge s((y \rightarrow x) \rightarrow x) \\
&=& (s(x \rightarrow y )\rightarrow s(y)) \wedge (s(y \rightarrow x) \rightarrow s(x)) \\
&=& ((s(x)\rightarrow s(y))\rightarrow s(y)) \wedge ((s(y) \rightarrow s(x)) \rightarrow s(x)) \\
&=& s(x)\vee s(y)
\end{eqnarray*}
(iv) Since $x^n\le s(x^n)$ and $x^n\le s(x)^n$ for every $x\in H$, by Proposition \ref{1.9}(iii), $1= x^n \vee y^n \le s(x^n)\vee s(y^n)$ and $1= x^n \vee y^n \le s(x)^n \vee s(y)^n.$ Thus, the equality holds.
\end{proof}

\begin{prop}\label{3.25}
Let $H$ be a basic hoop  algebra and $s:H \rightarrow H$ be a square root on $H$. For every $x\in H$:
\begin{itemize}
	\item[{\rm (i)}] If $s(x)$ is $\wedge$-irreducible, then $x$ is $\wedge$-irreducible.

	\item[{\rm (ii)}] If $s$ is onto, then the converse  of (i) is also true.

	\item[{\rm (iii)}] If $H$ is a $\vee$-hoop algebra and $s(x)$ is $\vee$-irreducible, then $x$ is $\vee$-irreducible.
\end{itemize} 
\end{prop}

\begin{proof}
(i) Assume that $s(x)$ is $\wedge$-irreducible. Let $a, b \in H$ such that $x = a \wedge b$. Then, by  Theorem \ref{3.12}(i), we have $s(x) = s(a \wedge b) = s(a) \wedge s(b)$. Since $s(x)$ is irreducible, we have $s(x) = s(a)$ or $s(x) = s(b)$. Now, by Theorem \ref{3.80}, we conclude that $x = a$ or $x = b$.\\
(ii) Let $x$ is $\wedge$-irreducible and $a,b\in H$ such that $s(x)=a \wedge b$. Since $s$ is onto, there are $u,v\in H$ such that $s(u)=a$ and $s(v)=b$. As $H$ is basic we have $s(x)=s(u) \wedge s(v) = s(u \wedge v)$. Now according to \ref{3.80}, $x=u \wedge v$. Since $x$ is $\wedge$--irreducible, hence $x=u$ or $x=b$. Therefore $s(x)=s(u)=a$ or $s(x)=s(v)=b$ and the result follows.\\
(iii) The proof is similar to  (i).
\end{proof}

\begin{rmk}\label{3.85} 
It is worth mentioning that if in the above proposition $H$ is not basic, but $s$ has the property that $s(x\odot y) = s(x) \odot s(y)$, for all $x,y\in H$,   then the proposition is true, as well. This is achieved by Proposition~\ref{3.10}(iii).
\end{rmk}

\begin{thm}\label{3.17}
Let $H$ be a bounded $\vee$-hoop algebra with $(DNP)$ and $s:H \rightarrow H$ be a square root on $H$. Then for every $x,y\in H$:
\begin{itemize}
	\item[{\rm (i)}] $s(x)''= s(x)=s(x'')$,
	
	\item[{\rm (ii)}] $s(x\odot y)  = (s(x)\odot s(y))\vee s(0)$. In particular,  $x \vee s(0) = s(x^2)$.
	
	\item[{\rm (iii)}] $s((s(0)')^2) = s(0)'$,
	
\item[{\rm (iv)}] $s(x')'\le s(x)$,
	
	\item[{\rm (v)}] $s(x)\vee s(y)= s(x\vee y)$,
	
	\item[{\rm (vi)}] If $x\le y$, then $s(\left [ x,y \right]) = \left [ s(x),s(y) \right]$. In particular, $s(H)= \left [ s(0),s(1) \right] $.  	
\end{itemize}	
\end{thm}

\begin{proof}

(i) The proof is straightforward.\\
(ii) By Theorem \ref{3.8}(vi), for all $x,y\in H$ we have
\begin{eqnarray*}
s(x\odot y) &=& s((x\odot y)'') \\
&=& s(((x\odot y)\rightarrow 0)\rightarrow 0) \\
&=& s((x\rightarrow (y\rightarrow 0) )\rightarrow 0) \\
&=& (s(x)\rightarrow (s(y)\rightarrow s(0)))\rightarrow s(0) \\
&=& ((s(x)\odot s(y))\rightarrow s(0))\rightarrow s(0)
\end{eqnarray*}
By  property $(DNP)$, we have $((s(x)\odot s(y))\rightarrow s(0))\rightarrow s(0) = (s(x)\odot s(y))\vee s(0).$ Thus, $s(x\odot y)=(s(x)\odot s(y))\vee s(0)$. Now we have $s(x\odot x)= (s(x)\odot s(x))\vee s(0)= x\vee s(0).$\\\
(iii) Using (ii), and Theorem \ref{3.13}(vi) we have $s((s(0)\rightarrow 0)\odot (s(0)\rightarrow 0))= (s(0)\rightarrow 0)\vee s(0) = s(0)\rightarrow 0.$\\
(iv) By Theorem \ref{3.13}(ix), the proof is straightforward.\\
(v) By Corollary \ref{2.21} and Theorem \ref{3.14}(iv), the proof is straightforward.\\
(vi) Suppose $x\le y$. By Theorem \ref{3.8}(xi), we have $s ([ x,y ]) \subseteq  [s(x),s(y)]$. Let $z\in H$ such that $s(x)\le z \le s(y)$. By part (ii), $z=s(z\odot z)$ since $ s(0)\le s(x)\le z$. Now we have $x=s(x)\odot s(x)\le z\odot z \le s(y)\odot s(y)=y$. Therefore, $z= s(z\odot z) \in s([x,y])$.
\end{proof}

\begin{defn}\label{3.18}
Let $H$ be a hoop  algebra and $n\in \mathbb{N}$. We say that $H$ has an $n$-th  root if there exists a unary operation $r_n:H \rightarrow H$ satisfying the following properties:

(NS1) For every $x\in H$, $r_n(x)^n = x$.

(NS2) For every $x, y \in H$, if $y^n \le x$, then $y \le r_n(x)$.\\

For the first root, we have $r_1(x) = x$, and for the square root, $r_2(x) = s(x)$.
Therefore, $r_2(x)\odot r_2(x) = s(x) \odot s(x) = x$.
\end{defn}

\begin{thm}\label{3.19}
Let $H$ be a hoop  algebra, and $r_n:H \rightarrow H$ and $r_m:H \rightarrow H$ be the $n$-th  and $m$-th roots in $H$ for $n,m\in \mathbb{N}$. Then for every $x, y \in H$, the following properties hold:
\begin{itemize}
\item[{\rm (i)}] $x\le r_n(x)$,

\item[{\rm (ii)}] If $x\le y$, then $r_n(x)\le r_n(y)$,

\item[{\rm (iii)}] $x\le r_n(x^n)$,

\item[{\rm (iv)}] $x\wedge y \le r_n(x\wedge y)\le r_n(x)\wedge r_n(y)$,

\item[{\rm (v)}] $r_n(x)\odot r_n(y) \le r_n(x\odot y)$,

\item[{\rm (vi)}] $r_n(x)\rightarrow r_n(y) = r_n(x \rightarrow y)$,

\item[{\rm (vii)}] $r_n(x)\le r_n(x^{n-1}) \rightarrow r_n(x^n)$,

\item[{\rm (viii)}]  $r_n(x)^m \le r_n(x^m)$,

\item[{\rm (ix)}] If $m\le n$, then $r_m(x)^m \le r_n(x)^m$,

\item[{\rm (x)}] If $a\in \id(H)$, then the operation $r_{n_{a}}:H[a] \rightarrow H[a]$ induced by  $r_n$, is an $n$-th  root on $H[a]$.
\end{itemize}	
\end{thm}
\begin{proof}
(i) By the condition $(NS1)$ of Definition~\ref{3.18}, we have $r_n(x)^n=x$. Therefore, $x\le r_n(x)$.\\
(ii) According to the condition $(NS2)$ of Definition \ref{3.18}, since $r_n(x)^n=x\le y$, we have $r_n(x)\le r_n(y)$.\\
(iii) Using $x^n=x\odot x\odot \cdots  \odot x$ and the condition $(NS2)$ of Definition \ref{3.18}, we have $x\le r_n(x^n)$.\\
(iv) The proof follows easily from (i).\\
(v) We have:
\begin{align*}
(r_n(x)\odot r_n(y))^n &= (r_n(x)\odot r_n(y)) \odot (r_n(x)\odot r_n(y)) \odot \cdots  \odot (r_n(x)\odot r_n(y)) \\
&= (\underbrace{r_n(x)\odot \cdots  \odot r_n(x)}_{n-\mbox{times}})\odot (\underbrace{r_n(y)\odot \cdots \odot r_n(y)}_{n-\mbox{times}})= x \odot y.
\end{align*}
Therefore, by (NS2) of Definition \ref{3.18}, $r_n(x)\odot r_n(y) \le r_n(x \odot y)$.\\
(vi) We have:
\begin{align*}
x \odot (r_n(x)\rightarrow r_n(y))^n &= (r_n(x))^n \odot (r_n(x)\rightarrow r_n(y))^n \\
&= (r_n(x)\odot (r_n(x)\rightarrow r_n(y))) \odot \cdots  \odot (r_n(x)\odot (r_n(x)\rightarrow r_n(y))) \\
&= (r_n(x)\wedge r_n(y)) \odot \cdots  \odot (r_n(x)\wedge r_n(y)) \\
&\le (\underbrace{r_n(y)\odot \cdots \odot r_n(y)}_{n-\mbox{times}})= y.
\end{align*}
Therefore, by Theorem \ref{1.4}(ii), $(r_n(x)\rightarrow r_n(y))^n \le x\rightarrow y$. In other words, by (NS2) of Definition \ref{3.18}, $r_n(x)\rightarrow r_n(y) \le r_n(x\rightarrow y)$.
Moreover, by (v), we have $r_n(x)\odot r_n(x\rightarrow y)\le r_n(x\odot(x\rightarrow y))= r_n(x\wedge y)\le r_n(y).$ 
Therefore, by Theorem \ref{1.4}(ii), $r_n(x\rightarrow y) \le r_n(x)\rightarrow r_n(y)$. Hence, $r_n(x\rightarrow y)= r_n(x)\rightarrow r_n(y)$.\\
(vii) The proof is straightforward.\\
(viii) Let $n, m \in \mathbb{N}$. Based on (v), we have:
 \[
r_n(x)^m = \underbrace{r_n(x) \odot r_n(x) 	\odot \cdots\odot r_n(x)}_{m-\mbox{times}}  \le r_n(x^m)
\]
(ix) According to (NS1) of  Definition \ref{3.18} we have:
\[ r_m(x)^m = x=r_n(x)^n = \underbrace{r_n(x)\odot \cdots  \odot r_n(x)}_{n-\mbox{times}}= r_n(x)^{n-m} \odot r_n(x)^m \le r_n(x)^m \]
(x) The proof is straightforward.
\end{proof}

\begin{exm}\label{3.55}
a) Considering Example \ref{3.2}(a), in $G_a$ we define the unary operation $r_n: G_a \rightarrow G_a$ as $r_n(a^\frac{u}{v}) = a^\frac{u}{nv}, \;\; \text{for all } a^\frac{u}{v} \in G_a.$ It is easily checked that the conditions of Definition \ref{3.18} hold. For instance, for condition (NS1), we have:
\begin{eqnarray*}
	r_n(a^\frac{u}{v})^n &=& r_n(a^\frac{u}{v}) \odot r_n(a^\frac{u}{v}) \cdots  \odot \cdots  \odot r_n(a^\frac{u}{v}) \\&=& a^\frac{u}{nv}\cdot a^\frac{u}{nv} \cdot \cdots  \cdot a^\frac{u}{nv} \\ &=& a^{\frac{u}{nv}+\frac{u}{nv}+ \cdots  + \frac{u}{nv} }\\ 
	&=& a^{\frac{u}{v}} 
\end{eqnarray*}
b) In accordance with Example \ref{3.2}(b), in $\bf{Goa}$, we define the unary operation $r_n: I \rightarrow I$ as $r_n(x) = x$ for every $x \in I$. 
Then, $r_n$ ia an $n$-th root on $\bf{Goa}$.\\
c) According to Example \ref{3.2}(c), in $\bf{Pra}$, we define the unary operation $r_n: I\rightarrow I$ as $r_n(x) = \sqrt[n]{x}$ for every $x \in I$. It can be easily checked that the conditions of Definition \ref{3.18} are satisfied. For instance, for condition (NS1), we have $r_n(x)^n = (\sqrt[n]{x})^n = x$. It is evident that the condition (NS2) also holds.
\end{exm}

\begin{prop}\label{3.20}
Let $H$ be a hoop  algebra and $r_n: H \rightarrow H$ be an $n$-th root on $H$ for $n \in \mathbb{N}$. The following statements hold:
\begin{itemize}
\item[{\rm (i)}] The $n$-th root is one to one.
\item[{\rm (ii)}] $r_n(1) = 1$.
\end{itemize}
\end{prop}

\begin{proof}
(i) Let $r_n: H \rightarrow H$ be an $n$-th root on $H$. For any $x, y \in H$, if $r_n(x) = r_n(y)$, then $r_n(x)^n  = r_n(y)^n$ and so $x = y.$\\
(ii) By Theorem \ref{3.19}(i), the proof is straightforward.
\end{proof}

\begin{thm}\label{3.21}
Let $H$ and $G$ be two hoop  algebras, and let $f: H \rightarrow G$ be an isomorphism. Then, for $n \in \mathbb{N}$, if $r_n: H \rightarrow H$ is an $n$-th root on $H$, then $G$  has an $n$-th root also.
\end{thm}

\begin{proof}
Suppose $a \in G$. Since $f$ is an isomorphism, there exists a unique $x \in H$ such that $f(x) = a$. We claim that $s_n: G \rightarrow G$ defined by $s_n(a) = f(r_n(x))$ is an $n$-th root on $G$. As $f$ is an  isomorphism, it is clear that $s_n$ is well-defined. Now for any $a \in G$, we have
\begin{eqnarray*}
s_n(a) \odot s_n(a) \odot \ldots \odot s_n(a) &=& f(r_n(x)) \odot f(r_n(x)) \odot \ldots \odot f(r_n(x)) \\
&=& f(r_n(x) \odot r_n(x) \odot \ldots \odot r_n(x)) \\
&=& f(x) \\
&=& a
\end{eqnarray*}
On the other hand, for any $a, b \in G$ such that $b \odot b \odot \ldots \odot b \leq a$, there exist  unique elements $x, y \in H$ such that $f(x) = a$ and $f(y) = b$. Therefore, $f(y) \odot f(y) \odot \ldots \odot f(y) \leq f(x)$. Since $f$ is an isomorphism, we have
\begin{align*}
f(y) \odot f(y) \odot \ldots \odot f(y) \leq f(x)&\Rightarrow f(y \odot y \odot \ldots \odot y) \leq f(x) \\
&\Rightarrow y \odot y \odot \ldots \odot y \leq x \\
&\Rightarrow y \leq r_n(x) \\
&\Rightarrow f(y) \leq f(r_n(x)) \\
&\Rightarrow b \leq s_n(a).
\end{align*}
\end{proof}

\begin{prop}\label{3.26}
Let $f: H \rightarrow G$ be a monomorphism of hoop algebras, and for $n \in \mathbb{N}$, let $r_n$ be an $n$-th root on $H$. Then, for each $x \in H$, the map $s_n: \text{im}(f) \rightarrow \text{im}(f)$ defined by $s_n(f(x)) = f(r_n(x))$ is an $n$-th root on $\text{im}(f)$.
\end{prop}

\begin{proof}
It is clear that $\text{im}(f)$ is a hoop subalgebra of $G$. Now, as $H$ is isomorphic to $\text{im}(f)$, according to Theorem \ref{3.21}, $s_n: \text{im}(f) \rightarrow \text{im}(f)$ is an $n$-th root on $\text{im}(f)$.
\end{proof}
\section{Applications of square roots in hoop  algebras}
In this section, a brief discussion on the application of the square root in hoop  algebras will be presented. Additionally, we will show that the class of all hoop  algebras with  square roots is a variety.

\begin{prop}\label{3.6}
	Let $H$ be a hoop  algebra, $F$  a proper filter on $H$ and $s:H \rightarrow H$ be a square root on $H$. Then
\begin{itemize}
\item[{\rm (i)}] $s(F)\subseteq F$,

\item[{\rm (ii)}] If $s(x)\odot s(y)= s(x\odot y)$, for all $x,y\in H$, then $s(F) = F$.
\end{itemize}
\end{prop}

\begin{proof}
(i) Using  condition F3 of Definition  	\ref{2.1} and   Theorem~\ref{3.3}(i), the result is obtained.\\	
(ii) Let $x \in F$. Then $s(x) \in s(F)$. By  the assumption, we have 	$x = s(x)\odot s(x) = s(x \odot x) \in s(F)$. Hence $F \subseteq s(F)$. According to  (i), we get the result.
\end{proof}

\begin{prop}\label{5.12}
Let $H$ be a hoop  algebra and $s:H \rightarrow H$ be a square root on $H$. If   $s(x)\odot s(y)= s(x\odot y)$, for all $x,y\in H$, then $s(H)$ is a subalgebra of $H$.
\end{prop}

\begin{proof}
It suffices to examine the conditions of Definition~\ref{2.1}. We define the operations $\odot$ and $\rightarrow$ on $s(H)$ to be the same as the operations $\odot$ and $\rightarrow$ in $H$. It is evident that $1 \in s(H)$. 	Let $x,y\in H$. It is clear that $x\odot y \in H$. Also, $s(x), s(y) \in s(H)$. Now, by  assumption of the proposition, we have $s(x)\odot s(y) = s(x \odot y) \in s(H)$. Therefore, $s(x)\odot s(y) \in s(H$. 	Furthermore, by   Theorem \ref{3.8}(vi), $s(x)\rightarrow s(y) = s(x \rightarrow y) \in s(H)$. Thus, $s(H)$ is a subalgebra of $H$.
\end{proof}

\begin{prop}\label{5.13}
Let $H$ be a hoop  algebra, $F$  a filter of $H$, and $s:H \rightarrow H$ is a square root on $H$. If $s(x)\odot s(y) = s(x\odot y)$, for every $x, y \in H$, then $s(F)$ is a filter of $H$ and $s(H)$.
\end{prop}

\begin{proof}
We examine the conditions of Definition \ref{2.1}. Since $1\in F$, we have $s(1)=1\in s(F)$. Therefore, the condition $F1$ holds. Let $s(x),s(y)\in s(F)$. This implies $x,y \in F$. Since $F$ is a filter, $x\odot y \in F$, hence $s(x\odot y) \in s(F)$. Now, according to the assumption, $s(x)\odot s(y) \in s(F)$. Therefore, the condition $F2$ holds. Next, let $s(x)\in s(F)$ and $y \in H$ such that $s(x)\le y$. By Theorem \ref{3.3}(i), $x\le y$. Since $F$ is a filter, $y\in F$, so $s(y)\in s(F)$. Therefore, the condition F3 holds. Hence, $s(F)$ is a filter of $H$. It is also easy to check that  $s(F)$ is a filter of $s(H)$.
\end{proof}

\begin{prop}\label{3.27}
Let $H$ be a bounded hoop  algebra, $F$  a filter of $H$, and $s:H \rightarrow H$ be a square root on $H$. Then the map $s_F: \frac{H}{F} \rightarrow \frac{H}{F}$ defined by $s_F(\frac{x}{F}) = \frac{s(x)}{F}$, for every $x \in H$, is a square root on $\frac{H}{F}$.
\end{prop}

\begin{proof}
The proof is straightforward. It is enough to check the conditions of the Definition \ref{3.1}. For condition $S1$, for every $x \in H$, according to Definition \ref{2.1}, we have $\frac{s(x)}{F} \odot \frac{s(x)}{F} = \frac{s(x)\odot s(x)}{F} = \frac{x}{F}$. For condition $S2$, let $x,y \in H$ such that $\frac{y}{F} \odot \frac{y}{F} \le \frac{x}{F}$. By the definition of order on $\frac{H}{F}$ one has $y^2 \ra x \in F \Rightarrow s(y^2 \ra x) \in F \Rightarrow  s(y^2) \ra s(x) \in F \Rightarrow \frac{s(y^2)}{F} \le \frac{s(x)}{F} \Rightarrow \frac{y}{F} \le \frac{s(x)}{F},$ and we are done. 
\end{proof}

\begin{thm}\label{5.14}
Let $H$ be a bounded hoop  algebra, $F$  a filter of $H$, and $s:H \rightarrow H$ be a square root on $H$. If  $s(x)\odot s(y)= s(x\odot y)$, for every $x,y\in H$, then 
\[
 s_{F}\left(\frac{H}{F}\right) \cong \frac{s(H)}{s(F)}.
\]
\end{thm}

\begin{proof}
By Proposition \ref{5.12}, $s(H)$ is a subalgebra of $H$. Also according to Proposition \ref{5.13}, $s(F)$ is a filter of $s(H)$ and then  $\frac{s(H)}{s(F)}$ is a hoop algebra. In addition,  by  Proposition \ref{3.27} we know that  $s\left(\frac{H}{F}\right)$ is a subalgebra of the hoop algebra $\frac{H}{F}$. Now,   define the map 
\[
f:s_{F}\left(\frac{H}{F}\right) \rightarrow \frac{s(H)}{s(F)}
\]
by the rule $f(s[x])=[s(x)]$ for every $s[x]\in s\left(\frac{H}{F}\right)$. We show that $f$ is an isomorphism. Let $x, y\in H$ and $s[x]=s[y]$. According to Theorem \ref{3.80}, we have $[x]=[y]$. Therefore, $x \rightarrow y \in F$ and $y \rightarrow x \in F$. By Definition \ref{2.1} and Proposition~\ref{3.6}, we have $s(x \rightarrow y) \in F$ and $s(y \rightarrow x) \in F$. Now, by Theorem \ref{3.8}(vi), we conclude that $s(x) \rightarrow s(y) \in F$ and $s(y) \rightarrow s(x) \in F$. Thus, $[s(x)]=[s(y)]$. Thus $f$ is well-defined. It is clear that $f$ is  onto. To see it is one to one, let $f(s[x]) = f(s[y])$. Then $\frac{s(x)}{s(F)} = \frac{s(y)}{s(F)}$  which is equivalent to  $s(x)\rightarrow s(y), s(y)\rightarrow s(x) \in s(F)$. Now we get  $x \rightarrow y, y \rightarrow x \in F$. So $[x]=[y]$ and then $s[x] = s[y].$\\	
 Furthermore, we have
\begin{align*}
		f(s[x] \rightarrow s[y]) & = f(s[x \rightarrow y]) \\ & = [s(x \rightarrow y)] \\ & = [s(x) \rightarrow s(y)] \\ & = [s(x)] \rightarrow [s(y)] \\ & = f(s[x]) \rightarrow f(s[y]),
\end{align*}
and
\begin{align*}
f(s[x] \odot s[y]) & = f(s[x \odot y]) \\ & = [s(x \odot y)] \\ & = [s(x) \odot s(y)] \\ & = [s(x)] \odot [s(y)] \\ & = f(s[x]) \odot f(s[y]).
\end{align*}
Finally, $f(s[1]) = [s(1)] = [1].$ Therefore, $f$ is an isomorphism.
\end{proof}

\begin{thm}\label{5.1}
Let $H$ be a hoop  algebra and $s:H \rightarrow H$ be a square root on $H$. Then $$\langle x\rangle= \langle s(x)\rangle= \langle s(x)^2\rangle = \cdots .$$
for every $x\in H$.
\end{thm}

\begin{proof}
By Theorem \ref{3.3}, we have $x\le s(x)$. Therefore, by Definition \ref{2.1}, $s(x) \in \langle x\rangle$. Hence, $\langle s(x)\rangle \subseteq \langle x\rangle$. 	
On the other hand, we know that $x=s(x)\odot s(x)\in \langle s(x)\rangle$. Then, $\langle x\rangle \subseteq \langle s(x)\rangle$. Therefore, $\langle x\rangle= \langle s(x)\rangle$. Consequently, $\langle x\rangle= \langle s(x)\rangle= \langle s(x)^2\rangle = \cdots.$
\end{proof}

\begin{thm}\label{5.4}
Let $H$ be a bounded $\vee$-hoop algebra, and $s:H \rightarrow H$ be a square root on $H$. If  $H$ is basic or   $s(x\odot y) = s(x) \odot s(y)$, for all $x,y\in H$,  then, a proper filter $F$ of $H$ is prime if and only if for every $x,y\in H$, $s(x)\vee s(y)\in F$, implies that $s(x)\in F$ or $s(y)\in F$.
\end{thm}

\begin{proof}
Let us first assume that $H$ is a basic hoop.	If $F$ is a prime filter, then clearly the implication holds. Now, assume that $F$ is a proper filter and $x\vee y \in F$. Then $s(x\vee y) \in F$. Since $H$ is basic, we have $s(x)\vee s(y) \in F$. By the assumption, $s(x)\in F$ or $s(y)\in F$. Therefore, by the condition $F2$ in Definition \ref{2.1}, $x\in F$ or $y\in F$. Thus, $F$ is prime.

Let us now assume that $s(x\odot y) = s(x) \odot s(y)$, for all $x,y\in H$. According to Proposition~\ref{3.10}(iii), the assertion follows. 
\end{proof}

\begin{prop}\label{5.6}
Let $H$ be a bounded hoop  algebra, $s:H \rightarrow H$  a square root on $H$, and $F$ be a proper filter of $H$. Then, for every $n\in \mathbb{N}$ and $x\in H$ such that $x\leq s(0)$ we have
	\begin{itemize}
		\item[{\rm (i)}]  $s(x)^n\notin F$.
		\item[{\rm (ii)}] $s^n(x)\notin F$ with $s^n(x) = \underbrace{s(\cdots (s(x))).}_{n-{\rm times}}$
	\end{itemize}
\end{prop}

\begin{proof}
(i) Let $x\in H$ such that $x\leq s(0)$, and suppose there exists $n\in \mathbb{N}$ such that $s(x)^n\in F$. Now, since $F$ is a filter and $s(x)^n = x\odot s(x)^{n-2}\le 0 \odot s(0)^{n-2}= 0$ , we have $0 \in F$, which contradicts the properness of $F$. Hence, for every $n\in \mathbb{N}$ and $x\in H$ such that $x\leq s(0)$, $s(x)^n\notin F$.\\	
(ii) This is clear.
\end{proof}

Let $H$ be a bounded hoop  algebra and $s:H \rightarrow H$ be a square root on $H$. For every $x\in H$, the set $S_x$ associated with $x$ is defined by
\[
S_x = \{s^n(x) : n\in \mathbb{N}\} = \{x,s(x),s^2(x),\cdots \}.
\]

\begin{thm}\label{5.9}
Let $H$ be a bounded hoop  algebra and $s:H \rightarrow H$ be a square root on $H$. Put	$S = \bigcup_{x\le s(0)} S_x$. Then, either $P=H-S$ is a maximal filter of $H$ or $\langle P \rangle = H$.
\end{thm}

\begin{proof}
We distinguish two cases.\\
1) $P$ is a filter of $H$.	If $P$ is not maximal, then there exist a proper filter $F$ of $H$ such that $P\subsetneq F$ and hence choose $x\in F-P$. As $x\in S$, there exist $y\in H$ and $k\in \mathbb{N}$ such that $y\le s(0)$ and $s^k(y)=x$. Thus, we have $x \le s^{k+1}(0)$. By Theorem~\ref{5.1}, we conclude that there is a positive integer $m$ for which $x^m = 0$, so $0 \in \langle x \rangle$ and then $\langle x \rangle = H.$ As $x\in F$, $\langle x \rangle\subseteq F$. Therefore, $F = H$, a contradiction. Therefor,  $P$ is a maximal filter and we are done. \\
2) The subset $P$ of $H$ is not a filter of $H$. In this case, $P\ne \langle P \rangle$ and we replace the filter $F$, as in (1), by   $\langle P \rangle$ and repeat the proof of last part to obtain  $\langle P \rangle = H$, as required.

\end{proof}
Let $H$ and $G$ be two hoop  algebras, and $f:H \rightarrow G$ be a hoop  homomorphism. Suppose $s:H \rightarrow H$ and $t:G \rightarrow G$ are square roots on $H$ and $G$, respectively. One has $f(s(x))\odot f(s(x))= f(x)$, for every $x\in H$.  According to condition $S2$, it is clear that $f(s(x))\le t(f(x))$.
\begin{defn}\label{4.6}
With the above notations, we say that $f$ \textit{preserves square roots} if $f(s(x))= t(f(x))$, for every $x\in H$. This means that the diagram 
\[\label{pulb. of alphaE}
	\SelectTips{cm}{}\xymatrix{ H \ar[d]_{s} \ar[r]^{f}&G \ar[d]^{t} \\
 H \ar[r]_{f} & G}
\]
is commutative. 
\end{defn}

\begin{thm}\label{4.7}
Let  $f:H \rightarrow G$ be a hoop  homomorphism and  $s:H \rightarrow H$ and $t:G \rightarrow G$ are square roots on $H$ and $G$, respectively. The  following statements are true.
\begin{itemize}
	\item[{\rm (i)}] $f$ preserves square roots if and only if the image of $f$, $Im(f)$, is closed under $t$.
	\item[{\rm (ii)}]  If $f$ is an isomorphism, then $t=f\circ s\circ f^{-1}$.
	\item[{\rm (iii)}] If $t(f(x))$ is $\wedge$-irreducible and $f$ is a one to one which preserves square roots, then $s(x)$ is $\wedge$-irreducible.
\item[{\rm (iv)}] If $H$ and $G$ are two $\vee$-hoop  algebras and $t(f(x))$ is $\vee$-irreducible and $f$ is  one to one which preserves square roots, then $s(x)$ is $\vee$-irreducible.
	
	Let $f$ preserves square roots. Then one has 
	\item[{\rm (v)}] 	If $s(x)\in N(H)$, then $t(f(x))\in N(G).$
    \item[{\rm (vi)}]	 If $s(x)\in D(H)$, then $t(f(x))\in D(G).$
	\item[{\rm (vii)}] If $s(x) \in \reg(H)$, then $t(f(x)) \in \reg(G).$
	
	If $f$ is a one to one, then the converse of items (v), (vi) and (vii) are true, as well.
\end{itemize}	

\end{thm}

\begin{proof}
The proof of  (i) and (ii) are similar to the proof of {\rm \cite[Theorem 3.14]{DvZa1}}. \\
(iii) Suppose $s(x) \ne 0$ and $a,b \in H$ such that $s(x) =a \wedge b$. Now we have $t(f(x))=f(s(x))= f(a\wedge b)= f(a)\wedge f(b)$. Since $t(f(x))$ is $\wedge$-irreducible, it follows that $f(s(x))=f(a)$ or $f(s(x))=f(b)$. Due to the one to one of $f$, we have either $s(x)=a$ or $s(x)=b$. Therefore, $s(x)$ is also $\wedge$-irreducible.\\	
(iv) The proof is similar to (iii).\\	
(v) We have the following implications
\[ 
s(x) \in N(H) \Longrightarrow s(x)^n =0 \Longrightarrow f(s(x)^n) =0 \Longrightarrow t(f(x))^n =0 \Longrightarrow t(f(x)) \in N(G).
\]
(vi) and (vii) are straightforward. Moreover, if $f$ is  one to one, then it can be easily checked that the converse of items (v), (vi) and (vii) hold.
\end{proof}

\begin{cor}\label{4.11}
Let $H$ and $G$ be two $\vee$-hoop algebras, and let $f:H\rightarrow G$ be a hoop  homomorphism. Suppose $s:H \rightarrow H$ and $t:G \rightarrow G$ are square roots on $H$ and $G$, respectively. If $G$ is a basic hoop  algebra and $f$ preserves square roots, then for every $x,y\in H$, $f(s(x\wedge y))= f(s(x)) \wedge f(s(y))$ and $f(s(x\vee y))= f(s(x)) \vee f(s(y))$.	
\end{cor}

\begin{proof}
The result follows from the application of Theorems \ref{3.12}(i) and \ref{3.14}(iii).
\end{proof}

\begin{defn}\label{4.1}
Let $H$ be a bounded hoop  algebra and $s: H \rightarrow H$ be a square root on $H$. If $s(0) = 0$, then $H$ is called {\em good}. 
\end{defn}	
As shown in Proposition \ref{3.10}, if for all $x, y \in H$, we have $s(x \odot y) = s(x) \odot s(y)$, or $s(x^2) = x$, then $H$ is good.

\begin{cor}\label{4.8}
Let $H$ and $G$ be two hoop  algebras, and $f:H\rightarrow G$ be a hoop  homomorphism. Suppose $s:H \rightarrow H$ and $t:G \rightarrow G$ are square roots on $H$ and $G$, respectively. If $H$ is good and $f$ be preserves square roots, then $G$ is also good.
\end{cor}

\begin{exm}\label{4.42}
All the hoop algebras defined in Examples~\ref{3.2}, are good.
\end{exm}

\begin{exm}\label{4.43}
a) Let $I=[0,1]$ be the unit closed interval of real numbers and $p\in \mathbb{N}$. Define the operations $\odot$ and $\rightarrow$ as $a \odot b = (\max(0, a^p + b^p - 1))^{\frac{1}{p}}, a \rightarrow b = \min(1, (1 - a^p + b^p)^{\frac{1}{p}}),$ for all  $a,b\in I$. Then, $\mathbf{L}_p = (I; \max, \min, \odot, \rightarrow, 0, 1)$ is a residuated Lattice, which is called a {\L}ukasiewicz generalized structure (see also~\cite{DPi,DBDPAJ}).
Assuming $p=1$, the resulting structure is called a {\L}ukasiewicz algebra, which is a hoop  algebra. In this  {\L}ukasiewicz algebra, we have $a \odot b = \max(0, a + b - 1), a \rightarrow b = \min(1, 1 - a + b).$ Define the unitary operation $s: I \rightarrow I$ by $s(x) = \frac{x+1}{2},$ for every $x\in I$, it is clear that $s(x) \in I$. Now,
\begin{align*} 
s(x) \odot s(x) &= \max(0, s(x) + s(x) - 1) \\
&= \max(0, x) = x
\end{align*}
Thus, the condition S1 in Definition \ref{3.1} holds. It can easily be shown that the condition S2 is also satisfied. For, suppose $x, y \in I$ such that $y \odot y \le x$. In this case, we have $y \odot y = \max(0, y + y - 1) \le x$. If $y \odot y = 0 \le x$, then according to  Theorem \ref{3.8}(v), $s(y) \odot s(y) \le s(y \odot y) \le s(x)$, implying $y \le s(x)$. If $y \odot y = y + y - 1 \le x$, then $2y - 1 \le x$, hence $y \le \frac{x+1}{2} = s(x)$. 
However, we have $s(0) = 0.5$. Therefore, the {\L}ukasiewicz algebra is not good. Finally, $s(x) \odot s(y) = s(x \odot y)$ does not hold.\\
b) As shown in Example \ref{1.51}(c), $\Gamma(G,u) =  ([0,u];\odot,\rightarrow,u)$ is a bounded hoop  algebra with the  operations $x\odot y = (x+y-u)\vee 0,  x\rightarrow y = (y-x+u)\wedge u.$ Let $G$ be a 2-divisible Abelian group. In this algebra, we define the square root as follows:
\[ 
s:\Gamma(G,u) \rightarrow \Gamma(G,u), \quad s(x)= \frac{x+u}{2}
\]
for every $x\in I$. It is clear that $s(x)\in \Gamma(G,u)$. Now, we have $s(x)\odot s(x) = (\frac{x+u}{2}+\frac{x+u}{2}-u)\vee 0 = (x+u-u) \vee 0 = x.$ 
Thus, the condition S1 of Definition \ref{3.1} holds. It can easily be verified that the condition S2 is also satisfied. Also, $s(0)=\frac{u}{2}$. Therefore, $\Gamma(G,u)$ is not good. Moreover,  $s(x)\odot s(y) =s(x\odot y)$ does not hold.
\end{exm}

We call a bounded hoop  algebra $H$ is regular if for every $x,y\in H$ such that $x \ne 0$ and $y \ne 0$, then $x \wedge y \ne 0$.
 
\begin{thm}\label{4.2}
Let $H$ be a bounded and good hoop  algebra. Then, for every $x\in H$:
\begin{itemize}
\item[{\rm (i)}] $x\wedge x'=0$,
\item[{\rm (ii)}] $s(x)\odot s(x')=0$,
\item[{\rm (iii)}] If, in addition,   $H$ satisfies $(DNP)$, then $x=s(x^2)$,
\item[{\rm (iv)}] $x\in \reg(H)$ if and only if $s(x)\in \reg(H)$,
\item[{\rm (v)}] If $H$ is basic, then $H$ is regular if and only if for every $x,y\in H$,  $s(x) \ne 0$ and $s(y) \ne 0$, imply $s(x \wedge y) \ne 0$.
\item[{\rm (vi)}] $s(x \odot y) \le (s(x)\odot s(y))''$.
If $H$ has $(DNP)$, then $s(x \odot y) = s(x) \odot s(y)$.
\end{itemize}
\end{thm}

\begin{proof}
By Theorem \ref{3.13}(iii), the proofs of (i) and (ii) are straightforward.\\
(iii) According to Theorem~\ref{3.17}(ii), the proof is straightforward.\\
(iv) Let us assume $x\in \reg(H)$. By Theorem \ref{3.8}(vi), we have:
\begin{align*}
s(x)'' &= (s(x)\rightarrow 0)\rightarrow 0 \\
&= (s(x)\rightarrow s(0))\rightarrow s(0) \\
&= s(x'') \\
&= s(x)
\end{align*}
Conversely, let's assume $s(x)\in \reg(H)$. By Theorem~\ref{3.8}(vi), we have:
\begin{align*}
s(x) &= s(x)''\\
&= (s(x)\rightarrow 0)\rightarrow 0 \\
&= (s(x)\rightarrow s(0))\rightarrow s(0) = s(x'')
\end{align*}
Now, by Theorem~\ref{3.80}, we have $x=x''$, hence $x\in \reg(H)$.\\
(v) Let $H$ be  a basic hoop. The \textit{only if} part is  evident. To prove the \textit{if} part, let $x,y\in H$ be two non-zero elements. Since $s(0) = 0$ we have $s(x)\ne 0$ and $s(y)\ne 0$. As $H$ is a basic hope we get $0\ne s(x)\wedge s(y) = s(x\wedge y)$. Therefore, $x\wedge y\ne 0$. \\
(vi) Due to Theorems \ref{3.13}(xiv) and \ref{3.17}(ii), the proof is clear.
\end{proof}

\begin{thm}\label{4.3}
Let $s:H \rightarrow H$ be a square root on a bounded hoop  algebra $H$. Then:
\begin{itemize}
\item[{\rm (i)}] If $H=\id(H)$, then $H$ is good.
\item[{\rm (ii)}] If $H$ satisfies $(DNP)$ and is good, then $H=\id(H)$.
\end{itemize}
\end{thm}

\begin{proof}
(i) By Proposition \ref{3.30}, the proof is straightforward.\\
(ii) Let's assume $H$ is good and $x\in H$. By Theorem \ref{3.13}(v) and $(DNP)$, we obtain $s(x)\le x' \rightarrow 0 = (x\rightarrow 0)\rightarrow 0 = x.$ Therefore, for every $x\in H$, $s(x) = x$. Thus, by Proposition \ref{3.30}, we have $H=I(H)$.
\end{proof}

\begin{thm}\label{4.4}
Let $H$ be a bounded $\vee$-hoop  algebra with $(DNP)$ and $s:H \rightarrow H$ be a square root on $H$. Then the following conditions are equivalent:
\begin{itemize}	
\item[{\rm (i)}] $H$ is good,
	
\item[{\rm (ii)}] $H$ is a Boolean algebra.
\end{itemize}
\end{thm}

\begin{proof}
According to  \cite[Theorem 3.13]{ABAK2}, the $\vee$-bounded hoop  algebra $(H;\odot,\rightarrow,0,1)$ with $(DNP)$ is equivalent to the $MV$-algebra $(H;\oplus,0)$, where $x\oplus y = x'\rightarrow y$. Now, by~\cite[Proposition 2.19]{Hol}, since $H$ is equipped with a square root, $H$ is a Boolean algebra if and only if it is good.
\end{proof}

\begin{rmk}\label{4.80}
Let $H$ be a bounded $\vee$-hoop algebra with $(DNP)$ and $s:H \rightarrow H$  a square root on $H$.\\
(i) If  $s(x)\odot s(y) = s(x\odot y)$ for every $x,y \in H$, holds, then by Proposition \ref{3.10}(iv) and Theorem \ref{4.4}, $H$ is a Boolean algebra.\\
(ii) In light of  Proposition~\ref{3.10}, Theorem~\ref{4.2}(iii) and Theorem~\ref{4.4}, for all $x \in H$, $x= s(x^2)$ if and only if $H$ is a Boolean algebra.  
\end{rmk}

A  hoop  algebra $H$ is said to  satisfies the \textit{ascending chain condition} $(ACC)$ if for every  sequence 
\[	
x_1\le x_2\le \cdots \le x_n \le x_{n+1} \le \cdots	
\]	  
of elements in $H$, there is an integer $k$ such that $x_i = x_k$ for all $i\geq k$.  $H$ is said to  satisfies the \textit{descending chain condition} $(DCC)$ if for every  sequence 
\[	
x_1\ge x_2\ge \cdots \ge x_n \ge x_{n+1} \ge \cdots	
\]	  
of elements in $H$, there is an integer $k$ such that $x_i = x_k$ for all $i\geq k$.

\begin{thm}\label{4.13}
Let $H$ be a finite bounded hoop  algebra and $s: H \rightarrow H$ be a square root on $H$. Then the following statements hold:
\begin{itemize}
    \item[{\rm (i)}] The map $s: H \rightarrow H$ is onto.
    \item[{\rm (ii)}] For every $x\in H$, $s(x)=x$, and thus $H$ is good.
\end{itemize}
\end{thm}
\begin{proof}
(i) It is clear that $s(H) \subseteq H$. Moreover, since $s$ is one to one  and $\lvert s(H) \lvert= \lvert H \lvert$, we have $s(H)= H$. Therefore, $s: H \rightarrow H$ is onto.\\
(ii) Suppose $x\in H$ such that $x < s(x)$. If $s(x) = s(s(x))$, then as $s$ is one to one, $x = s(x)$, which is a contradiction. Thus, $s(x) < s(s(x))= s^2(x)$. Similarly, $s^2(x) < s^3(x)$. Hence, we obtain an ascending sequence of elements in $H$ as follows:
\[ 
x < s(x) < s^2(x) < \ldots 
\]
By the assumption of the finiteness of $H$, there exists $n \in \mathbb{N}$ such that $s^n(x)= s^{n+1}(x) = \ldots$. Therefore, by Theorem \ref{3.80}, we have $x = s(x)$, which is a contradiction. Thus, for every element in $H$, $s(x)=x$. Therefore, $H$ is good.
\end{proof}

\begin{prop}\label{4.14}
Let $H$ be a bounded hoop algebra and $s: H \rightarrow H$ be a square root on $H$. If $H$ satisfies  $(ACC)$, then for every $x \in H$, $x^2 = x$.
\end{prop}

\begin{proof}
Similar to the proof of Theorem \ref{4.13}(ii), for every $x \in H$, we have $s(x)=x$. Therefore, by Proposition~\ref{3.30}, for every $x \in H$, $x^2 = x$.
\end{proof}
\begin{rmk}\label{4.40}
	It is know that idempotent hoops are the \{$\wedge$,$\ra$,1\}-subreducts of Heyting algebras (see~{\rm \cite[Example 1.11]{BlFe2}}). So if a hoop algebra $H$ is  finite or it satisfies $(ACC)$, then the operation $\odot$ coincides with the operation binary meet. In fact, for all $x,y \in H$, we have  $x\odot y \ge x\odot (x \wedge y) = x\odot (x\odot(x\ra y)) = x\odot(x\ra y) = x \wedge y$. The inequality  $x\odot y \le x\wedge y$ always is true.
\end{rmk}
\begin{color}{blue}
\end{color}

\begin{thm}\label{Godel algebra}
Let $H$ be a bounded hoop algebra and $s: H \rightarrow H$  a square root on $H$. Let $H$ be finite or  it satisfies $(ACC)$. 
Then one has 
\begin{itemize}
	\item[{\rm (i)}] If $H$ is $\vee$-hoop with $(DNP)$, then $H$ is a Bollean algebra.
	
	\item[{\rm (ii)}] If  $H$ is  totally ordered, then $H$‌ has the structure of G\"{o}del algebra. 
\end{itemize}
\end{thm}

\begin{proof}	
(i) If $H$ is $\vee$-hoop with $(DNP)$, then by using Theorem \ref{4.13} and Proposition \ref{4.14}, $H$ is good, hence is a Bollean algebra according to Theorem~\ref{4.4}.\\
(ii) Let $H$ be totally ordered. Based on the assumption  and Remark \ref{4.40}, we have  $x\odot y = x \wedge y$, for all  $x,y \in H$. If $x \le y$, it is clear that $x \rightarrow y = 1$. Let us now assume $y< x$.  If $x\le  x \rightarrow y$ then  we have $y = x\wedge y =x\odot (x \rightarrow y) =x \wedge (x \rightarrow y)= x$ which is a contradiction. If $x \rightarrow y\le x$ we have  $y = x\wedge y = x\odot ( x \rightarrow y) =  x \wedge (x \rightarrow y) =x \rightarrow y$. Thus the implication is the G\"{o}del implication and $H$‌ has the structure of G\"{o}del algebra.
\end{proof}

\begin{prop}\label{4.16}
Let $H$ be a cancellative and bounded hoop  algebra, and let $s: H \rightarrow H$ be a square root on $H$. If $H$ satisfies $(DCC)$, then $H=\{1\}$.
\end{prop}

\begin{proof}
Assume $1 \ne x \in H$ such that $x < s(x)$. According to Theorem \ref{1.4}(iii), we have the following sequence:
\[	
 \ldots < (s(x))^3 < (s(x))^2 = x < s(x)
\]
By assumption, there exists a natural number $n$ such that $s(x)^{n+1} = s(x)^n = \ldots$. Therefore, by Proposition \ref{2.17} and Theorem \ref{3.80}, we have $s(x) = 1$, which is a contradiction. Thus, for every element $x\in H$, we have $s(x) = x$. By Proposition \ref{3.30}, for every $x \in H$, $x^2 = x$. Since  $H$ is cancellative we get $x=1$.  
\end{proof}
	
\begin{defn}\label{4.18}
Let $H$ be a bounded hoop  algebra and $s: H \rightarrow H$ be a square root on $H$. If $s(0) = s(0)'$, then $H$ is called \textit{strict}.
\end{defn}

\begin{exm}\label{4.19}
(a) In Example \ref{4.43}(a), we observed that in the {\L}ukasiewicz algebra, $s(0) = 0.5$. Therefore, the {\L}ukasiewicz algebra is not good. Now, we have
$s(0)' = s(0) \rightarrow 0 = 0.5 \rightarrow 0 = \min(1, 1 - 0.5 + 0) = 0.5 = s(0)$. Therefore, {\L}ukasiewicz algebra is strict.\\
(b) In Example \ref{4.43}(b), we observed that in the hoop algebra 
 $\Gamma(G,u)$, 
 $s(0) = \frac{u}{2}$. Therefore, $\Gamma(G,u)$ is not good. Now, we have $s(0)' = s(0) \rightarrow 0 = \frac{u}{2} \rightarrow 0 = (0 - \frac{u}{2} + u) \wedge u = \frac{u}{2} \wedge u = \frac{u}{2} = s(0).$ Therefore, the hoop algebra $\Gamma(G,u)$ is strict.
\end{exm}

\begin{prop}\label{4.21}
Let $H$ be a bounded hoop  algebra and $s: H \rightarrow H$ be a square root on $H$. Assume $H$ is good. Then the following conditions are equivalent:
\begin{itemize}
\item[{\rm (i)}] $H$ is strict,
\item[{\rm (ii)}] $H = \{1\}$.
\end{itemize}
\end{prop}
\begin{proof}
(i) $\Rightarrow$ (ii). Assume $H$ is strict. Since $H$ is good, $0 = s(0) = s(0)' = 0' = 1$.\\
(ii) $\Rightarrow$ (i). This is obvious.
\end{proof}

\begin{prop}\label{4.23}
Let $H$ be a bounded hoop  algebra and $s: H \rightarrow H$ be a square root on $H$. Assume that $H$ is strict. Then for all $x \in H$ we have 
\begin{itemize}
\item[{\rm (i)}] $s(x)' \le s(0)$,
\item[{\rm (ii)}] For every $x \in B(H)$, if $s(0) \le x$, then $x = 1$.
\end{itemize}
\end{prop}
\begin{proof}
(i) We have $0 \le x$, for every $x \in H$. Therefore, $s(0) \le s(x)$. Thus, by Proposition \ref{1.5}(ii),  $s(x)' \le s(0)' = s(0)$.\\
(ii) Let $x \in B(H)$ and $s(0) \le x$. By Proposition~\ref{1.5}(ii),  $x' \le s(0)' = s(0)$. Then, by  Lemma~\ref{4.22}(i) and Theorem~\ref{1.4}(xiii), we have $x' = (x')^2 \le s(0) \odot s(0) = 0$. Hence, $x = 1$.
\end{proof}

\begin{prop}\label{4.24}
Let $H$ and $G$ be two bounded hoop algebras, and $f: H \rightarrow G$ be a hoop  homomorphism. Suppose $s: H \rightarrow H$ and $t: G \rightarrow G$ are square roots on $H$ and $G$, respectively. If $H$ is strict and $f$ is a \textit{preserves square roots}, then $G$ is also strict.
\end{prop}
\begin{proof}
We have $t(0)= t(f(0))=f(s(0))= f(s(0)')=f(s(0))'= t(f(0))' =t(0)'.$
\end{proof}

In 2003, R. B\u{e}lohl\`{a}vek showed that the variety of all residuated lattices with square roots is a variety. We demonstrate that the class of all bounded hoop  algebras with square roots is also a variety (see \cite[Corollary 3(2)]{Rbe}).

\begin{thm}\label{4.17}
The class of all bounded hoop  algebras with square roots is a variety.
\end{thm}
\begin{proof}
Let $\mathcal{V}$ be the variety of all bounded hoop  algebras $(H;\odot,\rightarrow,0,1)$ with square roots. Also, let $\mathcal{W}$ be the variety of all hoop algebras $(G;\odot,\rightarrow,s,0,1)$ along with the following conditions:

(sq1) $s(x)\odot s(x)= x$,

(sq2) $s((y\odot y)\vee x) \wedge y = y$,

(sq3) $s(x')= s(x)\rightarrow s(0)$.

Assume $H\in \mathcal{V}$. Let $s:H \rightarrow H$ be a square root on $H$. We show that $H\in \mathcal{W}$. Since $H$ is a hoop  algebra, conditions H1 to H4, and then  (sq1) and (sq3) hold for $H$.  It suffices to verify condition (sq2). By Theorems \ref{3.14}(iii) and \ref{3.8}(v), one has
\[	
s((y\odot y)\vee x) \ge s(y\odot y)\vee s(x) \ge (s(y) \odot s(y)) \vee s(x) \ge y\vee s(x) \ge y
\]
Therefore, $s((y\odot y)\vee x)\wedge y = y$. Hence, $H\in \mathcal{W}$.

Conversely, assume $G\in \mathcal{W}$. Since $G$ satisfies conditions H1 to H4, it is a hoop  algebra. We only need to check condition S2 from Definition \ref{3.1}. Suppose $y\odot y \le x$. From condition (sq2), we obtain $y= s((y\odot y)\vee x) \wedge y = s(x)\wedge y$. Thus, $y\le s(x)$. Therefore, $s$ satisfies condition S2, and hence $G\in \mathcal{V}$. Thus, $\mathcal{V}=\mathcal{W}$.
\end{proof}

\bibliographystyle{amsplain}

 \end{document}